\documentclass[11pt]{amsart}
\usepackage{amsmath,amssymb,amscd,amsfonts,graphicx,latexsym,epsfig,psfrag,url,color,xr,mathtools,enumitem,transparent,float,fullpage,hyperref,marginnote,kantlipsum}
\usepackage[all]{xy}
\usepackage[foot]{amsaddr}

\renewcommand{\emph}[1]{{\bf{#1}}}

\newtheorem{theorem}{Theorem}[section]

\newtheorem{lemma}[theorem]{Lemma}

\theoremstyle{definition}
\newtheorem{definition}[theorem]{Definition}
\theoremstyle{remark}
\newtheorem{remark}[theorem]{Remark}

\newtheorem{example}[theorem]{Example}
\theoremstyle{plain}
\newcommand{\thistheoremname}{}
\newtheorem{genericthm}[theorem]{\thistheoremname}

\newtheorem*{genericthm*}{\thistheoremname}
\newenvironment{namedthm*}[1]
  {\renewcommand{\thistheoremname}{#1}%
   \begin{genericthm*}}
  {\end{genericthm*}}
  
\newcommand\cC{\mathcal{C}}

\newcommand\cM{\mathcal{M}}

\newcommand\cR{\mathcal{R}}

\newcommand\cU{\mathcal{U}}

\newcommand\SDT{\mathcal{SDT}}
\newcommand\SWC{\mathcal{SWC}}

\newcommand{\bN}{\mathbb{N}}
\newcommand{\bR}{\mathbb{R}}

\newcommand{\bC}{\mathbb{C}}
\newcommand{\bZ}{\mathbb{Z}}

\newcommand{\bCP}{\mathbb{CP}}

\newcommand\bn{\mathbf{n}}
\newcommand\bx{\mathbf{x}}

\newcommand\bz{\mathbf{z}}
\newcommand\bzero{\mathbf{0}}

\newcommand{\on}{\operatorname}

\renewcommand\Re{\on{Re}}

\newcommand\pt{{\on{pt}}}

\newcommand{\Fuk}{\on{Fuk}}
\newcommand{\comp}{C^2}

\renewcommand{\comp}{{\on{comp}}}
\newcommand{\seam}{{\on{seam}}}
\newcommand{\incom}{{\on{in}}}

\newcommand{\inte}{{\on{int}}}
\newcommand{\mk}{{\on{mark}}}

\renewcommand{\root}{{\on{root}}}

\newcommand{\new}{{\on{new}}}

\newcommand{\node}{{\on{node}}}
\newcommand{\spec}{{\on{spec}}}

\newcommand\qu{/\kern-.7ex/} 
\newcommand\lqu{\backslash \kern-.7ex \backslash}

\newcommand{\ol}{\overline}

\newcommand{\sr}{\stackrel}

\newcommand{\wt}{\widetilde}

\newcommand{\eps}{\epsilon}

\def\lra{\longrightarrow}

\newcounter{qcounter}

\newcommand\quotient[2]{
        \mathchoice
            {
                \text{\raise1ex\hbox{$#1$}\Big/\lower1ex\hbox{$#2$}}%
            }
            {
                #1\,/\,#2
            }
            {
                #1\,/\,#2
            }
            {
                #1\,/\,#2
            }
    }

\newcommand\quoti[2]{
                \text{\raise1ex\hbox{$#1$}/\lower1ex\hbox{$\scriptstyle#2$}}
  }

\newcommand\quot[2]{
                \text{\raise1ex\hbox{$#1\!\!$}/\lower1ex\hbox{$\!\scriptstyle#2$}}
  }

\newcommand\quo[2]{
                \text{\raise.8ex\hbox{$\scriptstyle#1\!$}/\lower.8ex\hbox{$\!\scriptstyle#2$}}
  }

\newcommand\qq[2]{
                \text{\raise.8ex\hbox{$#1\!$}/\lower.8ex\hbox{$#2$}}
}

\begin{document}

\title{Moduli spaces of witch curves topologically realize the 2-associahedra}
\author{Nathaniel Bottman}
\address{School of Mathematics, Institute for Advanced Study,
1 Einstein Dr, Princeton, NJ 08540}
\email{\href{mailto:nbottman@math.ias.edu}{nbottman@math.ias.edu}}

\begin{abstract}
For $r \geq 1$ and $\bn \in \bZ_{\geq0}^r\setminus\{\bzero\}$, we construct the compactified moduli space $\ol{2\cM}_\bn$ of witch curves of type $\bn$.
We equip $\ol{2\cM}_\bn$ with a stratification by the 2-associahedron $W_\bn$, and prove that $\ol{2\cM}_\bn$ is compact
and metrizable.
In addition, we show that the forgetful map $\ol{2\cM}_\bn \to \ol\cM_r$ to the moduli space of stable disk trees is continuous and respects the stratifications.
\end{abstract}

\maketitle

\section{Introduction}

In \cite{b:2ass}, the author constructed a collection of abstract polytopes (in particular, posets) called 2-associahedra.
There is a 2-associahedron $W_\bn$ for every $r\geq 1$ and $\bn \in \bZ_{\geq0}^r\setminus\{\bzero\}$, and they were introduced to model degenerations in the configuration space $2\cM_\bn$ of \textbf{stable witch curves}, whose interior parametrizes configurations of $r$ vertical lines in $\bR^2$ with $n_i$ marked points on the $i$-th line up to translations and positive dilations.
By identifying $\bR^2 \cup \{\infty\} \simeq S^2$, we can also view an element of $2\cM_\bn$ as a configuration of marked circles on $S^2$, where all the circles intersect at the south pole, up to M\"{o}bius transformations; both views are depicted in the following figure:

\begin{figure}[H]
\centering
\def\svgwidth{0.3\columnwidth}
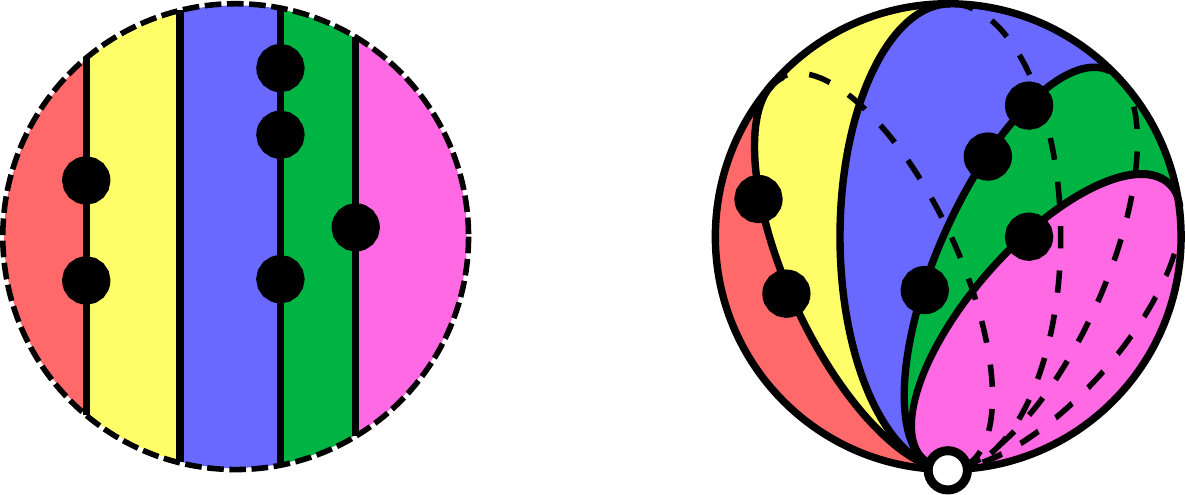
\label{fig:witch_ball}
\end{figure}

The purpose of this paper is to construct the compactified configuration space $\ol{2\cM}_\bn$, and to validate the construction of both $W_\bn$ and $\ol{2\cM}_\bn$ via the following main result:

\begin{theorem}
\label{thm:main}
For any $r \geq 1$ and $\bn \in \bZ^r_{\geq0}\setminus\{\bzero\}$, $\ol{2\cM}_\bn$ can be given the structure of a compact
metrizable space stratified by $W_\bn$.
The forgetful map $W_\bn \to K_r$ to an associahedron can be upgraded to a continuous map $\ol{2\cM}_\bn \to \ol\cM_r$ to the moduli space of stable disk trees that respects the stratifications.
\end{theorem}

\noindent This result is an important step toward the author's goal of defining a symplectic $(A_\infty,2)$-category \textsf{Symp}, in which the objects are certain symplectic manifolds and $\hom(M,N) \coloneqq \Fuk(M^-\times N)$,
where $\Fuk$ denotes the Fukaya category of a symplectic manifold.
Indeed, $(\ol{2\cM}_\bn)$ form the domain moduli spaces involved in the structure maps in \textsf{Symp}.
More progress toward the construction of \textsf{Symp} is described in \cite{b:sing, bw:compactness, b:2ass}.

In \S D of \cite{ms:jh}, McDuff--Salamon equip the compactified moduli space $\ol\cM_r(\bC)$ of
$r$-marked stable
genus-0 curves with a topology by including it into a product of $\bCP^1$'s via a collection of cross-ratio maps.
This is the obvious approach to try here, too, but the author was unable to make this technique work in this context.
Instead, we adapt the techniques from \S5 of the same book, in which McDuff--Salamon equipped the compactified moduli space of stable maps into a symplectic manifold with a topology in which the convergent sequences are those
that
Gromov-converge.
While this necessitates a certain amount of topological overhead in our setting, an advantage is that it will be straightforward to adapt the current work to the setting of witch maps when such a result is needed.


The construction of $\ol{2\cM}_\bn$ generalizes several earlier constructions of domain moduli spaces for pseudoholomorphic quilts.
Specifically, \cite{mw} and \cite{mww} construct several configuration spaces of disks decorated by interior circles, with marked points on the boundary and interior circles.
The data of such a configuration is equivalent to a configuration of circles with marked points on a sphere, as illustrated in the figure above.
\cite{mww} defines $\ol\cR^d$, $\ol\cR^{d,0}$, $\ol\cR^{d,e}$, and $\ol\cR^{d,0,0}$ (called associahedra, multiplihedra, biassociahedra, and bimultiplihedra, though the author of the current paper would rather reserve these names for the underlying posets).
In the notation of the current paper, these configuration spaces are $\ol{2\cM}_d$, $\ol{2\cM}_{d0}$, $\ol{2\cM}_{de}$, and $\ol{2\cM}_{d00}$.


\subsection{An example of Gromov convergence for witch curves}
\label{ss:coda}

As a coda to the introduction, we illustrate and motivate the definition of Gromov convergence in $\ol{2\cM}_\bn$ by an example.
For $\eps \in (0,\tfrac 1 2)$, consider the following configuration in $\ol{2\cM}_{10010}$
(here pictured with $\eps = 2/5$):

\vspace{-0.5em}
\begin{figure}[H]
\centering
\def\svgwidth{0.325\columnwidth}
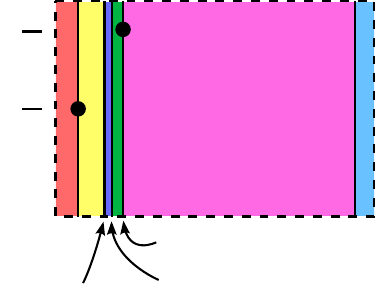
\end{figure}
\vspace{-1em}

\noindent In the limit as $\eps \to 0$, all lines but the right-most collide; the two marked points also collide.
We resolve these collisions using the well-known technique of \emph{soft rescaling}: whenever a marked point collides with a line (and in particular, with another marked point), we zoom in on the collision with just enough magnification that the colliding objects occupy a ``window'' of unit size.
If, in this zoomed-in view, there are still colliding objects, we again rescale, and so on inductively.

A decision must be made about what to do when lines without marked points collide; here, we have decided to remember the fashion in which such lines collide, a choice
that
is motivated by considerations of pseudoholomorphic quilts.
We implement this strategy by keeping track of the positions of the lines as points in $\bR$ and performing soft rescaling on these configurations in parallel with our soft rescalings of the configurations of lines and points in $\bR^2$.

Finally, we are ready to demonstrate soft rescaling for the family pictured above.
This is shown in the following figure, where the left-most view is the original configuration, and the remaining configurations are the rescaled views.
The arrows indicate that a configuration is produced by rescaling at the point that the arrow points to, with magnification labeling the arrow.
In the bottom of the figure, we show the soft rescalings of the configurations of the line positions in $\bR$.

\begin{figure}[H]
\centering
\def\svgwidth{1.0\columnwidth}
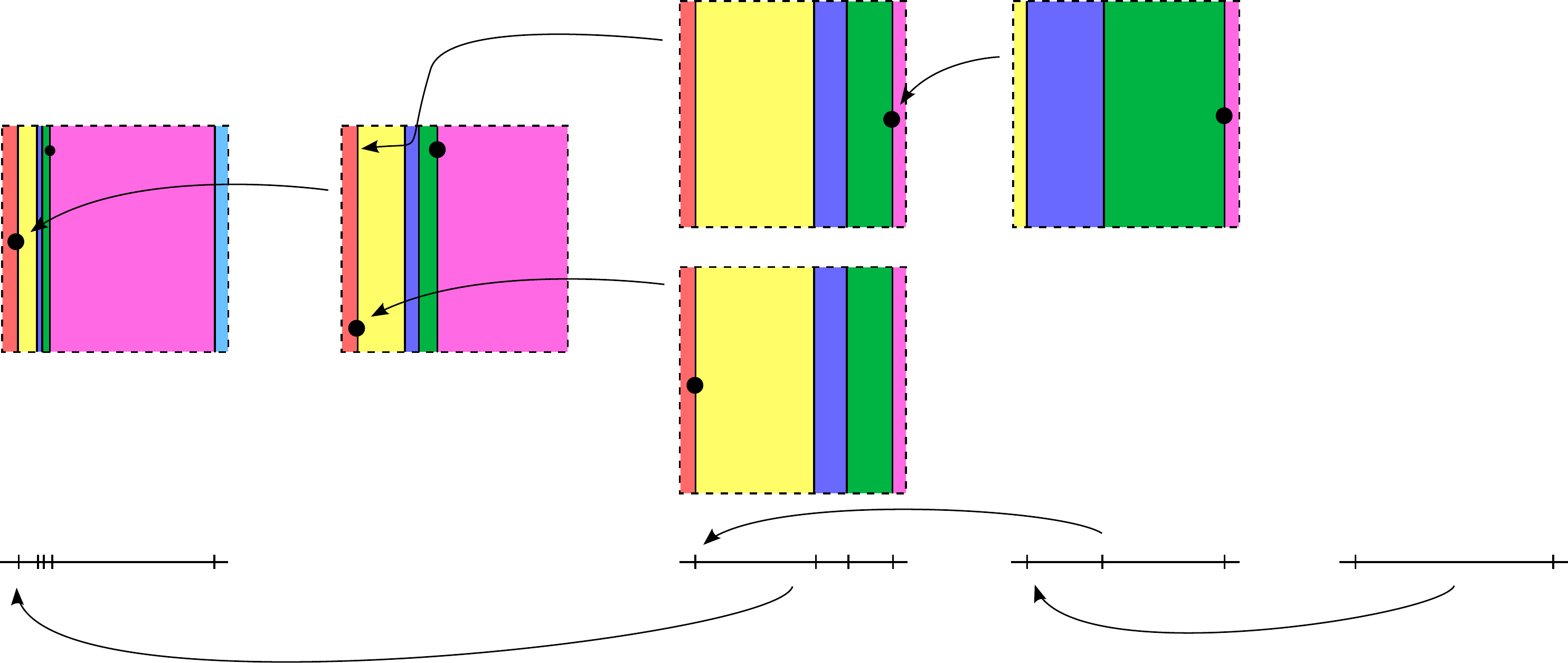
\end{figure}

We show the $\eps \to 0$ limit of this family in the following figure.
On the right, we show an equivalent view: the planes with marked vertical lines are replaced with spheres with marked circles.
In the tree of decorated spheres, the ``nodal points'' --- where the south pole of one sphere is attached to one of the circles on another sphere --- indicate that
we produced the upper sphere via a sequence of further rescalings of the rescalings we used to produce the lower sphere, and that these further rescalings were centered at the attachment point.

\begin{figure}[H]
\centering
\def\svgwidth{1.0\columnwidth}
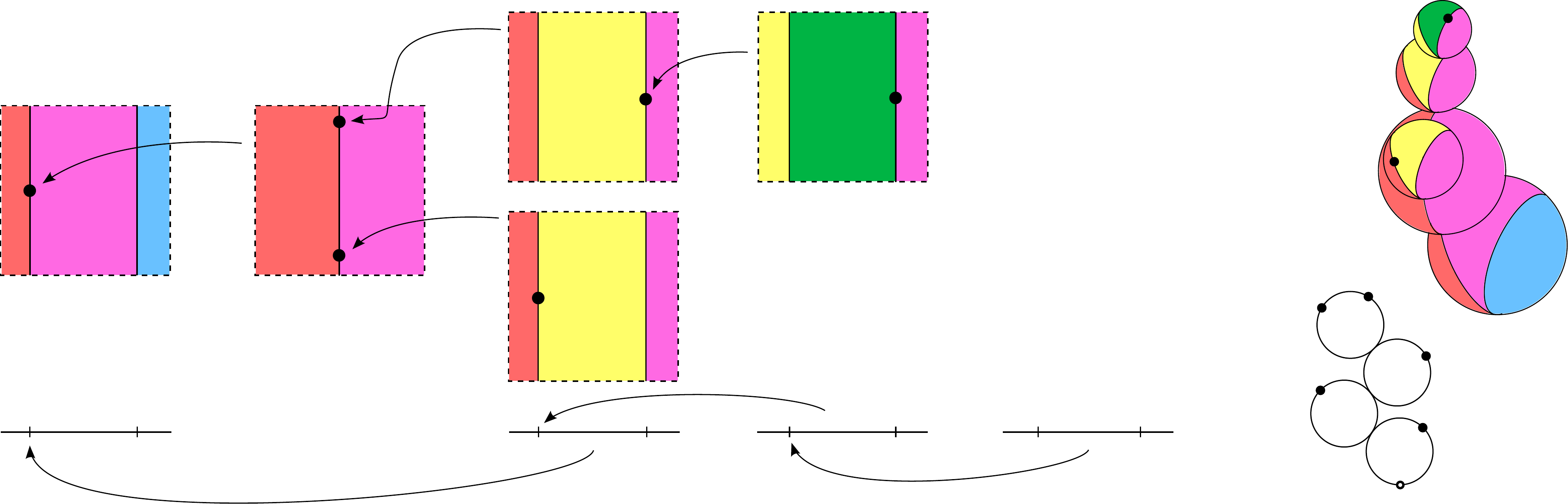
\end{figure}



\subsection{Acknowledgments}

The author thanks Mohammed Abouzaid, Helmut Hofer, Paul Seidel, and Katrin Wehrheim for their continuing support of this project.
A question of Satyan Devadoss led the author to the realization that when multiple unmarked lines collide, one should remember the relative speeds at which the lines collide.
This paper was written while the author was a Member and Schmidt Fellow at the Institute of Advanced Study and a Postdoctoral Research Fellow at Princeton University; he was supported by NSF Mathematical Sciences Postdoctoral Research Fellowship DMS-1606435.

\section{Construction of \texorpdfstring{$\ol{2\cM}_\bn$}{2Mn}}

In this section we prove Thm.~\ref{thm:main}.
Specifically, in \S\ref{ss:def} we construct $\ol{2\cM}_\bn$; in \S\ref{ss:compact} we show that every sequence in $\ol{2\cM}_\bn$ has a Gromov-convergent subsequence; in \S\ref{ss:unique_limits} we show that a Gromov-convergent sequence has a unique limit; and in \S\ref{ss:topology} we define a topology on $\ol{2\cM}_\bn$ in which the convergent sequences are the Gromov-convergent ones.

Before we construct $\ol{2\cM}_\bn$, we recall
the compactified moduli space $\ol\cM_r$ of disks with $r$ ``input'' and 1 ``output'' boundary marked points.
This moduli space is well-known:
see, for instance, \S4 of \cite{liu}, or Thm.~3.10 of \cite{DFHV}, which relies on ideas from \cite{FM}.
Nearly all of the results we describe below for $\ol{2\cM}_\bn$ have analogues for $\ol\cM_r$ --- in particular, $\ol\cM_r$ can be given a topology in which the convergent sequences are the Gromov-convergent ones, and with this topology it is compact,
metrizable, and stratified by $K_r$.
We will make use of these analogous results throughout this paper, mentioning them as we need them.
We now recall the definition of $\ol\cM_r$, making use of the notation for rooted ribbon trees from \S2, \cite{b:2ass}.
After the definition, we will give some motivation and illustrate some of the notation for rooted ribbon trees.

By convention, $\ol\cM_1 = \ol{2\cM}_{(1)} = \pt$.

\begin{definition}
\label{def:SDT}
A \emph{stable disk tree with $r \geq 2$ input marked points} is a pair $\bigl(T, (\bx_\rho)_{\rho \in V_\inte(T_s)}\bigr)$, where:
\begin{itemize}
\item $T$ is a stable rooted ribbon tree (RRT) with $r$ leaves.

\item For $\rho \in V_\inte(T)$, $\bx_\rho \in \bR^{\#\incom(\rho)}$ is a tuple satisfying $x_{\rho,1} < \cdots < x_{\rho,\#\incom(\rho)}$.
\end{itemize}
We say that two stable disk trees $\bigl(T,(\bx_\rho)\bigr)$, $\bigl(T',(\bx'_\rho)\bigr)$ are {\bf isomorphic} if there is an isomorphism of RRTs $f\colon T \to T'$ and a function $V_\inte(T) \to G_1\colon \rho \mapsto \phi_\rho$
(where $G_1$ is the reparametrization group $\bR \rtimes \bR_{>0}$ acting on $\bR$ by translations and positive dilations)
such that:
\begin{gather}
x'_{f(\rho),i} = \phi_\rho(x_{\rho,i}) \:\:\forall\:\: \rho \in V_\inte(T).
\end{gather}

We denote by $\SDT_r$ the collection of stable disk trees with $r$ input marked points, and we define the \emph{moduli space of stable disk trees with $r$ input marked points $\ol\cM_r$} to be the set of isomorphism classes of stable disk trees of this type.
For any stable RRT $T$ with $r$ leaves, define the corresponding \emph{strata} $\SDT_{r,T} \subset \SDT_r$, $\ol\cM_{r,T} \subset \ol\cM_r$ to be the set of all stable disk trees (resp.\ isomorphism classes thereof) of the form $\bigl(T,(\bx_\rho)\bigr)$.
We say that a stable disk tree is {\bf{smooth}} if its underlying RRT $T$ has only one interior vertex; we denote a smooth stable witch curve by the tuple $\bx \in \bR^r$ associated to the root.
\null\hfill$\triangle$
\end{definition}

\begin{remark}[motivation for Def.~\ref{def:SDT} from \S\ref{ss:coda}]
On the left in the following figure is the limit in $\ol{2\cM}_{10010}$ from \S\ref{ss:coda}.
\begin{figure}[H]
\centering
\def\svgwidth{0.45\columnwidth}
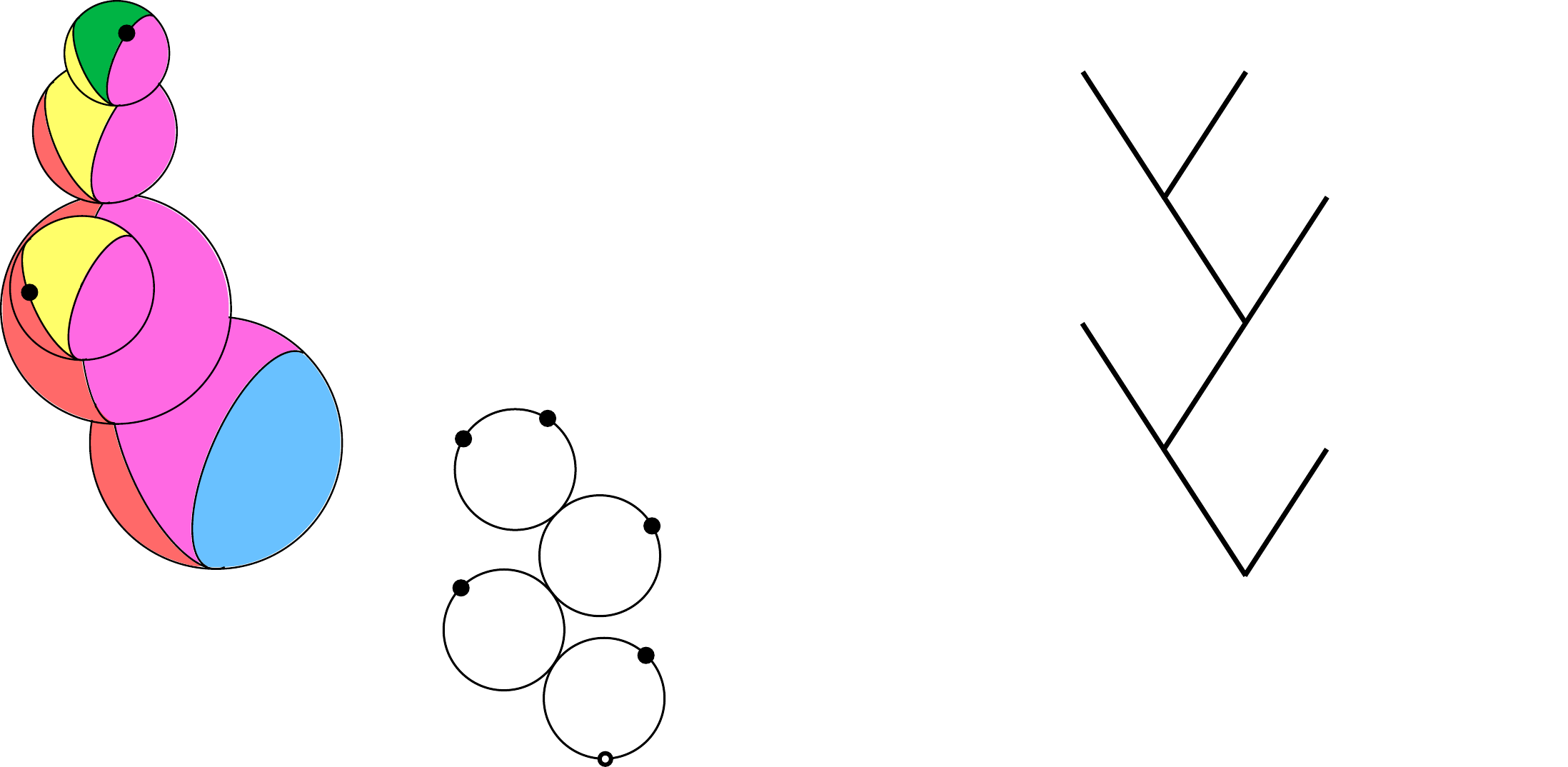
\end{figure}
\noindent As we explained in that subsection, this limit consists of a tree of decorated spheres, together with a datum (shown here as a tree of disks with boundary marked points) which tracks the seam positions.
This datum can be formulated as a stable disk tree $\bigl(T,(\bx_\rho)\bigr)$ as in Def.~\ref{def:SDT}, and on the right of this figure we show the RRT $T$.
Its interior vertices $\rho, \sigma, \tau, \upsilon$ correspond to the disks appearing on the left side of the figure, and the leaves correspond to the marked points (except for the bottommost marked point, which does not correspond to a vertex of $T$).
Each interior vertex $\rho,\sigma,\tau,\upsilon$ is assigned a tuple $\bx_\rho,\bx_\sigma,\bx_\tau,\bx_\upsilon$, which we think of as the $x$-positions of the seams.
\null\hfill$\triangle$
\end{remark}

\begin{example}
We recall a figure from \S2, \cite{b:2ass}, which illustrates some RRT notation in the case of a particular stable RRT:

\begin{figure}[H]
\centering
\def\svgwidth{0.8\columnwidth}
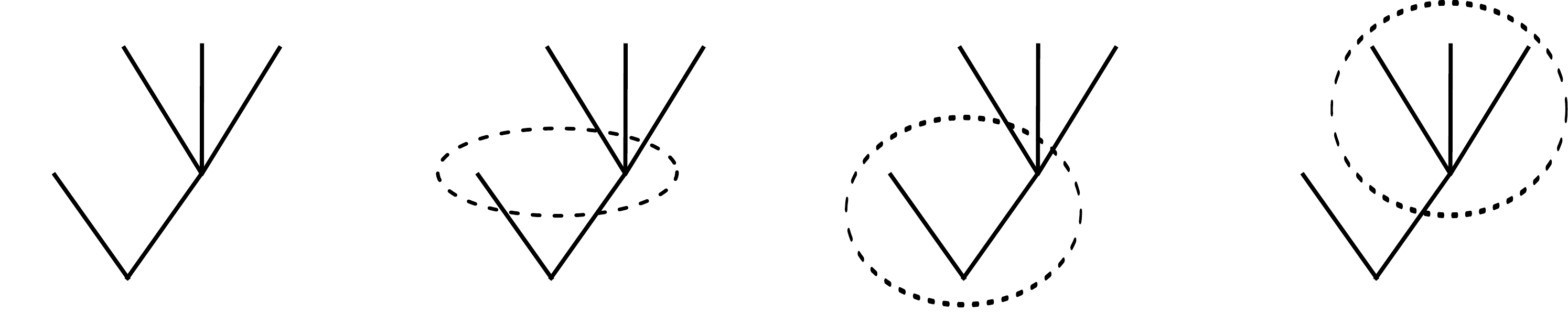
\label{fig:RRT_example}
\end{figure}
\noindent The leaves of $T$ are denoted $\lambda_1^T,\ldots,\lambda_4^T$, and the root (which is not considered a leaf) is denoted $\rho_\root^T$.
The interior vertices --- denoted $T_\inte$ or $V_\inte(T)$ --- are the non-leaf vertices.
The tree is oriented toward the root, and the set of incoming neighbors of a vertex $\rho$ is denoted $\incom(\rho)$.
(In fact, $\incom(\rho)$ inherits a total ordering from the ribbon structure of $T$.)
For distinct $\rho,\sigma \in T$, $T_{\rho\sigma}$ is the subtree consisting of those vertices $\tau$ such that the path from $\rho$ to $\tau$ passes through $\sigma$.
Finally, we denote $T_\sigma \coloneqq T_{\rho_\root\sigma}$.

This RRT is stable, because for every $\rho \in T_\inte$, $\incom(\rho)$ has at least two elements.
\null\hfill$\triangle$
\end{example}


\subsection{Definition of \texorpdfstring{$\ol{2\cM}_\bn$}{2Mn} as a set, and Gromov convergence}
\label{ss:def}

In this subsection we define \emph{stable witch trees}, isomorphism classes of which comprise $\ol{2\cM}_\bn$.
Throughout, we will denote by $\bR^2 \cup \{\infty\}$ the one-point compactification of $\bR^2$ (so $\bR^2 \cup \{\infty\} \cong S^2$).
We will make use of the reparametrization group $G_2 \coloneqq \bR^2 \rtimes \bR_{>0}$ acting on $\bR^2$ by translations and positive dilations.
This action of $G_2$ on $\bR^2$ extends to an action on $\bR^2 \cup \{\infty\}$, by defining $\phi(\infty) \coloneqq \infty$ for every $\phi \in G_2$.
There is a projection $p\colon G_2 \to G_1$, defined by sending $\bigl((x,y) \mapsto (ax+b_1,ax+b_2)\bigr) \in G_2$ to $\bigl(x \mapsto ax+b_1\bigr) \in G_1$.
We will overload notation and also denote by $p$ the projection $\bR^2 \to \bR^1$ onto the first factor.
Finally, we freely use the stable tree-pair notation introduced in \S3, \cite{b:2ass}.
In that paper, stable tree-pairs were denoted $T_b \sr{f}{\to} T_s$; here, we will use the notation $T_b \sr{\pi}{\to} T_s$.

\begin{definition}
\label{def:SWC}
A \emph{stable witch curve of type $\bn \in \bZ_{\geq0}^r\setminus\{\bzero\}$} is a triple
\begin{align}
\Bigl(2T=(T_b \sr{\pi}{\to} T_s), (\bx_\rho)_{\rho \in V_\inte(T_s)}, (\bz_\alpha)_{\alpha \in V_\comp(T_b)}\Bigr),
\end{align}
where:
\begin{itemize}
\item $2T$ is a stable tree-pair of type $\bn$.

\item For $\rho \in V_\inte(T)$, $\bx_\rho \in \bR^{\#\incom(\rho)}$ is a tuple satisfying $x_{\rho,1} < \cdots < x_{\rho,\#\incom(\rho)}$.

\item For $\alpha \in V_\comp(T_b)$, $\bz_\alpha \subset \bR^2$ is a collection
\begin{align}
\bz_\alpha = \left(z_{\alpha,ij} = (x_{\alpha,i},y_{\alpha,ij}) \:\left|\: {{\incom(\alpha) = (\beta_1,\ldots,\beta_{\#\incom(\alpha)}),}
\atop
{1\leq i\leq \#\incom(\alpha), \: 1 \leq j \leq \#\incom(\beta_i)}}\right.\right)
\end{align}
satisfying $x_{\alpha,1} < \cdots < x_{\alpha,\#\incom(\alpha)}$ and $y_{\alpha,i,1} < \cdots < y_{\alpha,i,\#\incom(\beta_i)}$ for every $i$.
Moreover, for $\alpha \in V_\comp^{\geq2}(T_b)$ we require $(x_{\alpha,1},\ldots,x_{\alpha,\#\incom(\alpha)}) = (x_{\pi(\alpha),1},\ldots,x_{\pi(\alpha),\#\incom(\pi(\alpha))})$.
\end{itemize}
We say that two stable witch curves $\bigl(2T,(\bx_\rho),(\bz_\alpha)\bigr)$, $\bigl(2T',(\bx'_\rho),(\bz'_\alpha)\bigr)$ are {\bf isomorphic} if there is an isomorphism of
stable
tree-pairs $2f\colon 2T \to 2T'$ and functions $V_\inte(T_s) \to G_1\colon \rho \mapsto \phi_\rho$ and $V_\comp(T_b) \to G_2\colon \alpha \mapsto \psi_\alpha$ such that:
\begin{gather}
z'_{f_b(\alpha),ij} = \psi_\alpha(z_{\alpha,ij}) \:\:\forall\:\: \alpha \in V_\comp(T_b),
\qquad
x'_{f_s(\rho),i} = \phi_\rho(x_{\rho,i}) \:\:\forall\:\: \rho \in V_\inte(T_s), \\
p(\psi_\alpha) = \phi_{\pi(\alpha)} \:\:\forall\:\: \alpha \in V_\comp^{\geq 2}(T_b). \nonumber
\end{gather}

\noindent We denote the collection of stable witch curves of type $\bn$ by $\SWC_\bn$, and we define the \emph{moduli space $\ol{2\cM}_\bn$ of stable witch curves of type $\bn$} to be the set of isomorphism classes of stable witch curves of this type.
For any
stable
tree-pair $2T$ of type $\bn$, define the corresponding \emph{strata} $\SWC_{\bn,2T} \subset \SWC_\bn$, $\ol{2\cM}_{\bn,2T} \subset \ol{2\cM}_\bn$ to be the set of all stable witch curves (resp.\ isomorphism classes thereof) of the form $\bigl(2T,(\bx_\rho),(\bz_\alpha)\bigr)$.
We say that a stable witch curve is {\bf{smooth}} if its underlying
stable
tree-pair $2T$ has the property that $V_\inte(T_s)$ and $V_\comp(T_b)$ each contain only one element; we denote a smooth stable witch curve by the pair
$(\bx,\bz) \in \bR^r\times\bR^{|\bn|+r}$
associated to the roots of $T_s$ resp.\ $T_b$.
\null\hfill$\triangle$
\end{definition}

\begin{remark}[motivation for Def.~\ref{def:SWC} from \S\ref{ss:coda}]
Once again, on the left in the following figure is the limit in $\ol{2\cM}_{10010}$ from \S\ref{ss:coda}.
\begin{figure}[H]
\centering
\def\svgwidth{0.65\columnwidth}
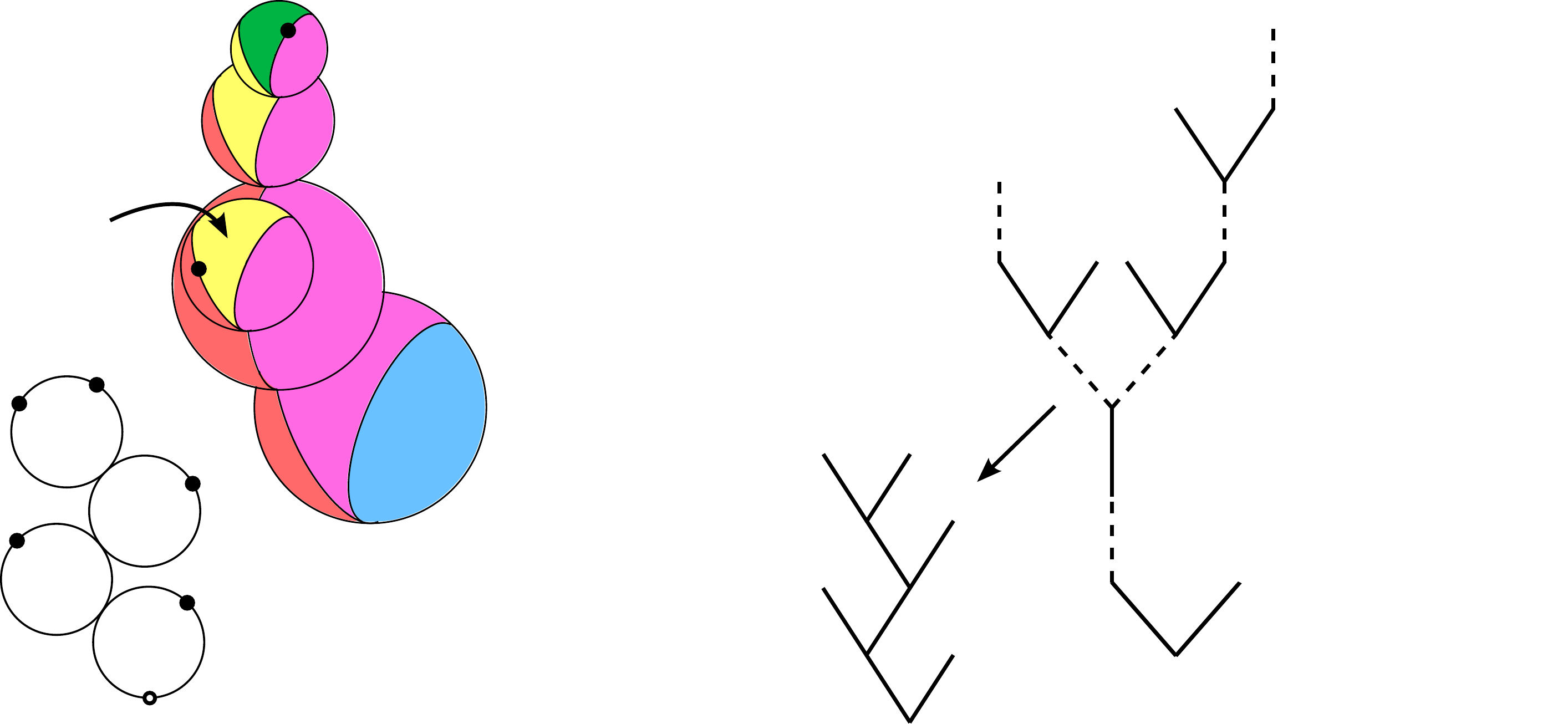
\end{figure}
\noindent This limit can be formulated as a stable witch curve $\bigl(2T,(\bx_\rho),(\bz_\alpha)\bigr)$ as in Def.~\ref{def:SWC}, and on the right of this figure we show the tree-pair $2T = T_b \to T_s$.
As explained in \S3, \cite{b:2ass}, the vertices of $T_b$ are partitioned as $V(T_b) = V_\comp(T_b) \sqcup V_\seam(T_b) \sqcup V_\mk(T_b)$ (``component vertices'', ``seam vertices'', and ``marked point vertices'').
The component vertices are labeled as $\alpha,\beta,\gamma,\delta,\eps$ in this figure, and they correspond to the spheres appearing on the left side of the figure.
Each component vertex $\alpha,\beta,\gamma,\delta,\eps$ is assigned a tuple $\bz_\alpha,\bz_\beta,\bz_\gamma,\bz_\delta,\bz_\eps$, which we think of as the positions of the special (marked and nodal) points.
Each dashed edge corresponds to a marked or nodal point.
\null\hfill$\triangle$
\end{remark}


\begin{example}
\label{ex:tree-pair_examples}
We recall a figure from \S3, \cite{b:2ass}, which illustrates some tree-pair notation in the case of a particular tree-pair $2T$:

\begin{figure}[H]
\centering
\def\svgwidth{0.95\columnwidth}
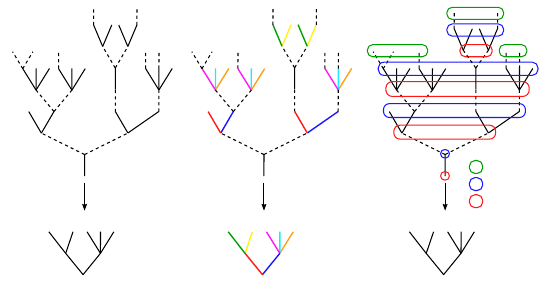
\label{fig:tree-pair_examples}
\end{figure}

\noindent $2T = T_b \sr{\pi}{\to} T_s$ consists of a ``bubble tree'', a ``seam tree'', and a map from the former to the latter.
Both $T_b$ and $T_s$ are RRTs, and $T_b$ has additional structure.
In particular, the vertices of $T_b$ are partitioned as $V_\comp(T_b)\sqcup V_\seam(T_b) \sqcup V_\mk(T_b)$, as shown on the right, and the edges alternate between solid and dashed ones.
The elements of $V_\mk(T_b)$ are denoted $\mu_{ij}^{T_b}$, as shown on the left, and the root of $T_b$ is denoted $\alpha_\root^{T_b} \in V_\comp(T_b)$.
The coherence map $\pi\colon T_b \to T_s$ is required to satisfy several conditions, as recorded in \S3, \cite{b:2ass}.
In the middle of this figure, we indicate how $\pi$ acts: we color the edges of $T_s$, and use those same colors to show which edges in $T_b$ are identified with the various edges of $T_s$.
Some edges in $T_b$ are contracted by $\pi$, which we indicate by using black.

This tree-pair is stable, because (1) $T_s$ is a stable RRT, and (2) for every $\alpha \in V_\comp(T_b)$ with $\incom(\alpha)$ denoted $(\beta_1,\ldots,\beta_k)$, either $k\geq2$ and there is a $\beta_i$ with $\#\incom(\beta_i)\geq1$, or $k=1$ and $\#\incom(\beta_1)\geq2$.
\null\hfill$\triangle$
\end{example}


If $\bigl(2T, (\bx_\rho), (\bz_\alpha)\bigr)$ is a stable witch curve and $\alpha \in V_\comp(T_b), \beta \in V_\comp(T_b) \cup (\mu_{ij})_{i,j}$ are distinct, then we define $z_{\alpha\beta} \in \bR^2 \cup \{\infty\}$ like so: Define $(\alpha=\gamma_1,\gamma_2,\ldots,\gamma_k=\beta)$ to be the
path from $\alpha$ to $\beta$.
If $\gamma_2$ is closer to the root than $\alpha$, then we define $z_{\alpha\beta}\coloneqq\infty$.
If $\gamma_2$ is the $i$-th incoming neighbor of $\alpha$ and $\gamma_3$ is the $j$-th incoming neighbor of $\gamma_2$, then we define $z_{\alpha\beta} \coloneqq z_{\alpha,ij}$.
For distinct $\rho \in V_\inte(T_s)$, $\sigma \in V(T_s)$, we define $x_{\rho\sigma}$ similarly.
For $\alpha, \beta$ as above, we define $x_{\alpha\beta} \coloneqq p(z_{\alpha\beta})$.
For $\alpha \in V_\comp(T_b)$ and $\rho \in V(T_s) \setminus \{\pi(\alpha)\}$, set $x_{\alpha\rho} \coloneqq x_{\pi(\alpha)\rho}$.
For $\alpha \in V_\comp^1(T_b)$, extend this definition by setting
\begin{align}
\label{eq:def_xalpharho_alpha_in_V1}
x_{\alpha\rho}
\coloneqq
\begin{cases}
x_{\alpha,1}, & \lambda_i \in (T_s)_{\pi(\alpha)}, \\
\infty, & \text{otherwise}.
\end{cases}
\end{align}
Finally, for any $\alpha \in V_\comp(T_b)$, we denote $z_{\alpha\mu_\infty^{T_b}} \coloneqq \infty$, $x_{\alpha\lambda_\infty^{T_s}} \coloneqq \infty$; here $\mu_\infty^{T_b}$ and $\lambda_\infty^{T_s}$ are formal
expressions,
rather than vertices in $T_b$ resp.\ $T_s$, which represent the fact that the root of the bubble tree and seam tree should be thought of as carrying a single ``output'' marked point.

We define the \textbf{set of nodal points} and \textbf{set of special points} of any interior vertex $\alpha$ like so:
\begin{gather}
Z_\alpha^\node \coloneqq (z_{\alpha\beta} \:|\: \beta \in V_\comp(T_b) \setminus \{\alpha\}) \subset \bR^2\cup\{\infty\},
\\
Z_\alpha^\spec \coloneqq (z_{\alpha\beta} \:|\: \beta \in (V_\comp(T_b)\cup (\mu_{ij})_{i,j} \cup \{\mu_\infty^{T_b}\}) \setminus \{\alpha\}) \subset \bR^2\cup\{\infty\}. \nonumber
\end{gather}

Before we define Gromov convergence for stable witch curves, we need two preliminaries: a way to express the property that two vertices in $V_\comp(T_b)$ correspond to two spheres attached via a nodal point, and a notion of surjection for
stable
tree-pairs.
The first notion is straightforward: for any
stable
tree-pair $T_b \to T_s$, we say that $\alpha, \beta \in V_\comp(T_b) \cup (\mu_{ij})_{i,j}$ are \textbf{contiguous} if the path from $\alpha$ to $\beta$ consists of $\alpha$, $\beta$, and a third vertex (necessarily in $V_\seam(T_b)$).
The second notion is less obvious.
If $2T'$ is the result of making a single move on $2T$ (in the sense of \S3.1, \cite{b:2ass}), then there are evident maps $T_b' \to T_b$, $T_s' \to T_s$.
Composing these maps inductively, we see that for any
stable
tree-pairs with $2T' < 2T$, there are induced maps $T_b' \to T_b$, $T_s' \to T_s$.
We call any map obtained in this fashion a \textbf{
stable
tree-pair surjection}.
Note that for any
stable
tree-pair surjection $2T' \to 2T$, the restriction $T_s' \to T_s$ to seam trees is an RRT surjection
as in \S2.1, \cite{b:2ass}
.

In the following definition, and throughout this paper, ``u.c.s.'' means ``uniformly on compact subsets''.
We refer to the notion of Gromov-convergence of a sequence $\bigl(T^\nu, (\bx_\rho^\nu)\bigr)$ of stable disk trees, which is similar to the notion of Gromov convergence of a sequence of stable genus-0 curves as in Def.\ D.5.1, \cite{ms:jh}.
The main difference with that notion is that for a sequence of stable disk trees to Gromov-converge, the maps $f^\nu\colon T^\nu \to T$ must be RRT homomorphisms.

\begin{definition} \label{def:2M_gromov_converge}
A sequence $\bigl(2T^\nu, (\bx_\rho^\nu), (\bz_\alpha^\nu)\bigr) \in \SWC_\bn$ is said to {\bf{Gromov-converge}} to $\bigl(2T, (\bx_\rho), (\bz_\alpha)\bigr)$ if the following conditions hold:
\begin{itemize}
\item $\bigl(T_s^\nu, (\bx_\rho^\nu)\bigr)$ Gromov-converges to $\bigl(T_s, (\bx_\rho)\bigr)$ via some $f^\nu\colon T_s \to T_s^\nu$ and $(\phi_\rho^\nu) \subset G_1$.

\item For $\nu$ sufficiently large, there is a
stable
tree-pair surjection $2f^\nu\colon 2T \to 2T^\nu$ covering $f^\nu\colon T_s \to T_s^\nu$ and a collection of reparametrizations $(\psi_\alpha^\nu)_{\alpha \in V_\comp(T_b)} \subset G_2$ such that the following hold:
\begin{itemize}
\item[] ({\sc restriction}) For $\alpha \in V_\comp^{\geq2}(T_b)$, $p(\psi_\alpha^\nu) = \phi_{\pi(\alpha)}^\nu$.

\item[] ({\sc rescaling}) If $\alpha,\beta \in V_\comp(T_b)$ are contiguous, and if $\nu_j$ is a subsequence such that $f_b^{\nu_j}(\alpha) = f_b^{\nu_j}(\beta)$, then the sequence $\psi_{\alpha\beta}^{\nu_j} \coloneqq (\psi_\alpha^{\nu_j})^{-1} \circ \psi_\beta^{\nu_j}$ converges to $z_{\alpha\beta}$ u.c.s.\ away from $z_{\beta\alpha}$.

\item[] ({\sc special point}) If $\alpha \in V_\comp(T), \beta \in V_\comp(T_b) \cup (\mu_{ij})_{i,j}$ are contiguous, and if $\nu_j$ is a subsequence such that $f_b^{\nu_j}(\alpha) \neq f_b^{\nu_j}(\beta)$, then:
\begin{align}
z_{\alpha\beta} = \lim_{j\to\infty} (\psi_\alpha^{\nu_j})^{-1}\Bigl(z^{\nu_j}_{f_b^{\nu_j}(\alpha)f_b^{\nu_j}(\beta)}\Bigr).
\hspace{0.5in}\triangle
\end{align}
\end{itemize}
\end{itemize}
\end{definition}

\noindent
Note that
if $\bigl(2T^\nu, (\bx_\rho^\nu), (\bz_\alpha^\nu)\bigr)$ Gromov-converges to $\bigl(2T, (\bx_\rho), (\bz_\alpha)\bigr)$ via $(\phi_\rho^\nu)$ and $(\psi_\alpha^\nu)$, and $(\wt\phi^\nu_\rho)_{\rho\in V_\inte(T_s^\nu)} \subset G_1$ and $(\wt\psi^\nu_\alpha)_{\alpha\in V_\comp(T_b^\nu)} \subset G_2$ are any sequences of reparametrizations satisfying
\begin{align}
p(\wt\psi_\alpha^\nu) = \wt\phi_{\pi^\nu(\alpha)}^\nu \:\:\forall\:\: \alpha \in V_\comp^{\geq2}(T_b^\nu),
\end{align}
then $\bigl(2T^\nu, (\wt\phi^\nu_\rho(\bx_\rho^\nu)), (\wt\psi^\nu_\alpha(\bz_\alpha^\nu))\bigr)$ Gromov-converges to $\bigl(2T, (\bx_\rho), (\bz_\alpha)\bigr)$ via $\bigl(\wt\phi_{f_s^\nu(\rho)}^\nu \circ \phi_\rho^\nu \bigr)_{\rho \in V_\inte(T_s)}$ and $\bigl(\wt\psi_{f_b^\nu(\alpha)}^\nu \circ \psi_\alpha^\nu \bigr)_{\alpha \in V_\comp(T_b)}$.

The following lemma shows that Gromov convergence in $\ol{2\cM}_\bn$ actually implies {\it{a priori}} stronger versions of the {\sc (rescaling)} and {\sc (special point)} axioms; for simplicity, we state it in the case that the surjection $2T \to 2T^\nu$ is fixed.

\begin{lemma} \label{lem:2M_gromov_conseq}
Suppose that $\bigl(\wt{2T}, (\bx_{\wt\rho}^\nu), (\bz_{\wt\alpha}^\nu)\bigr) \subset \SWC_\bn$ Gromov-converges to $\bigl(2T,(\bx_\rho), (\bz_\alpha)\bigr)$ via $2f\colon 2T \to \wt{2T}$, $(\phi_\rho^\nu)$, and $(\psi_\alpha^\nu)$.
Then the following properties hold.
\begin{itemize}
\item[] ({\sc Rescaling'}) For any distinct $\alpha,\beta \in V_\comp(T_b)$ with $f_b(\alpha) = f_b(\beta)$, the sequence $\psi_{\alpha\beta}^\nu$ converges to $z_{\alpha\beta}$ u.c.s.\ away from $z_{\beta\alpha}$.

\item[] ({\sc Special point'}) For any $\alpha \in V_\comp(T_b)$, $\beta \in V_\comp(T_b) \cup (\mu_{ij})_{i,j}$ with $f_b(\alpha) \neq f_b(\beta)$, the equality $z_{\alpha\beta} = \lim_{\nu\to\infty} (\psi_\alpha^\nu)^{-1}(z_{f_b(\alpha)f_b(\beta)})$ holds.
\end{itemize}
\end{lemma}

\begin{proof}
\begin{itemize}
\item[] ({\sc Rescaling'}) Denote by $(\alpha=\gamma_1,\ldots,\gamma_k=\beta)$ the vertices in $V_\comp(T_b)$ through which the path from $\alpha$ to $\beta$ passes, and note that $f_b(\alpha)=f_b(\beta)$ implies $f_b(\alpha=\gamma_1)=f_b(\gamma_2)=\cdots=f_b(\gamma_k=\beta)$.
We prove the claim by induction on $k$.
The $k=2$ case is exactly {\sc (rescaling)}.
Suppose that we have proven the claim up to and including some particular $k$; we now must prove the claim in the case that the path from $\alpha$ to $\beta$ has length $k+1$.
By assumption, $\psi_{\alpha\gamma_k}^\nu$ converges to $z_{\alpha\gamma_k}$ u.c.s.\ away from $z_{\gamma_k\alpha}$ and $\psi_{\gamma_k\beta}^\nu$ converges to $z_{\gamma_k\beta}$ u.c.s.\ away from $z_{\beta\gamma_k}$.
The fact that $(\gamma_1,\ldots,\gamma_{k+1})$ does not intersect itself implies $z_{\gamma_k\beta}\neq z_{\gamma_k\alpha}$, $z_{\alpha\gamma_k} = z_{\alpha\beta}$, and $z_{\beta\gamma_k} = z_{\beta\alpha}$, so it follows that $\psi_{\alpha\beta}^\nu = \psi_{\alpha\gamma_k}^\nu\circ\psi_{\gamma_k\beta}^\nu$ converges to $z_{\alpha\beta}$ u.c.s.\ away from $z_{\beta\alpha}$.

\medskip

\item[] ({\sc special point'}) Denote by $(\alpha=\gamma_1,\ldots,\gamma_k=\beta)$ the vertices in $V_\comp(T_b) \cup (\mu_{ij})_{i,j}$ through which the path from $\alpha$ to $\beta$ passes.
We prove the claim by induction on $k$.
The $k=2$ case is exactly {\sc (special point)}.
Suppose that we have proven the claim up to and including some particular $k$; we now must prove the claim in the case that the path from $\alpha$ to $\beta$ includes $k+1$ elements of $V_\comp(T_b) \cup (\mu_{ij})_{i,j}$.
If $f_b(\gamma_{k-1}) = f_b(\beta)$, then the claim follows from the inductive hypothesis:
\begin{align}
z_{\alpha\beta} = z_{\alpha\gamma_{k-1}} = \lim_{\nu\to\infty} (\psi_\alpha^\nu)^{-1}(z_{f(\alpha)f(\gamma_{k-1})}^\nu) = \lim_{\nu\to\infty} (\psi_\alpha^\nu)^{-1}(z_{f_b(\alpha)f_b(\beta)}^\nu).	
\end{align}
Otherwise, we use the inductive hypothesis and the inequality $z_{\gamma_{k-1}\beta} \neq z_{\gamma_{k-1}\alpha}$:
\begin{align}
z_{\alpha\beta}
=
z_{\alpha\gamma_{k-1}}
=
\lim_{\nu\to\infty} \psi_{\alpha\gamma_{k-1}}^\nu\bigl((\psi_{\gamma_{k-1}}^\nu)^{-1}(z_{f_b(\gamma_{k-1})f_b(\beta)}^\nu)\bigr)
=
\lim_{\nu\to\infty} (\psi_\alpha^\nu)^{-1}(z_{f_b(\alpha)f_b(\beta)}^\nu).
\end{align}
\end{itemize}
\end{proof}

Next, we prove
an alternate version
of {\sc(special point)} in the case of a Gromov-convergent sequence of smooth stable witch curves.

\begin{lemma} \label{lem:V1_x-limits}
Suppose that $(\bx^\nu,\bz^\nu) \subset \SWC_\bn$ Gromov-converges to $\bigl(2T,(\bx_\rho),(\bz_\alpha)\bigr)$ via $(\phi_\rho^\nu)$ and $(\psi_\alpha^\nu)$.
For any $\alpha \in V_\comp^1(T_b)$ and $\lambda_i \in V(T_s)\setminus V_\inte(T_s)$, the equality $x_{\alpha\lambda_i} = \lim_{\nu\to\infty} p\bigl((\psi_\alpha^\nu)^{-1}\bigr)(x^\nu_i)$ holds.
\end{lemma}


\begin{proof}
\noindent {\sc Step 1:} {\it If $\beta \in V_\comp(T_b)$ is closer to the root than $\alpha \in V_\comp(T_b)$, and we denote $((\psi_\alpha^\nu)^{-1} \circ \psi_\beta^\nu)(z) \eqqcolon a^\nu z + b^\nu$, then $\lim_{\nu\to\infty} a^\nu = \infty$.}

\medskip

\noindent By ({\sc rescaling'}), $(\psi_\alpha^\nu)^{-1}\circ\psi_\beta^\nu$ converges to $\infty$ u.c.s.\ away from $z_{\beta\alpha} \in \bR^2$.
The equality $\lim_{\nu\to\infty} a^\nu = \infty$ follows.

\medskip

\noindent {\sc Step 2:} {\it We prove the claim in the case that $\alpha$ is further from the root from a vertex $\beta \in V_\comp^{\geq 2}(T_b)$ and closer to the root than a vertex $\gamma \in V_\comp^{\geq 2}(T_b)$.}

\medskip


\noindent
First, suppose that $\lambda_i$ does not lie in $(T_s)_{\pi(\alpha)}$, and choose $\beta$ to be the closest vertex to $\alpha$ having the property just mentioned.
The stability of $2T$ implies that some $\mu_{i'j} \in V_\mk(T_b)$ lies in $(T_b)_\alpha$.
{(\sc Special point')} now yields the equality $\lim_{\nu\to\infty} p\bigl((\psi_\alpha^\nu)^{-1}(z_{i'j}^\nu)\bigr) = x_{\alpha,1} \in \bR$, hence
\begin{align}
\lim_{\nu\to\infty} p\bigl((\psi_\alpha^\nu)^{-1}\bigr)(x_{i'}^\nu) = x_{\alpha,1} \in \bR.
\end{align}
By our choice of $i'$ and $\beta$, $x_{\pi(\beta)\lambda_i}$ and $x_{\pi(\beta)\lambda_{i'}}$ are distinct elements of $\bR\cup\{\infty\}$, hence by {\sc (restriction)} and {\sc (special point')} we have
\begin{align}
\lim_{\nu\to\infty} p\bigl((\psi_\beta^\nu)^{-1}\bigr)(x_i^\nu)
=
x_{\pi(\beta)\lambda_i}
\neq
x_{\pi(\beta)\lambda_{i'}}
=
\lim_{\nu\to\infty} p\bigl((\psi_\beta^\nu)^{-1}\bigr)(x_{i'}^\nu).
\end{align}
Step 1, along with the last two displayed (in)equalities, yields $\lim_{\nu\to\infty} p\bigl((\psi_\alpha^\nu)^{-1}\bigr)(x_i^\nu) = \infty$,
which, by \eqref{eq:def_xalpharho_alpha_in_V1}, is equal to $x_{\alpha\lambda_i}$.

A similar argument (using $\gamma$ in place of $\beta$) proves that if $\lambda_i$ lies in $(T_s)_{\pi(\alpha)}$, then $p\bigl((\psi_\alpha^\nu)^{-1}\bigr)(x_i^\nu)$ converges to $x_{\alpha,1}$.


\medskip

\noindent {\sc Step 3:} {\it We prove the claim when $\alpha$ does not satisfy the hypothesis of Step 2.}

\medskip

\noindent In this case, $\pi(\alpha)$ must lie in $(\lambda_i)_i \cup \{\rho_\root^{T_s}\}$.
Suppose $\pi(\alpha) = \rho_\root^{T_s}$.
If $r = 1$, the claim clearly holds.
Otherwise, choose $\gamma$ to be the element of $V_\comp^{\geq 2}(T_b)$ closest to $\alpha$.
For every $i'$, {\sc(restriction)} and {\sc(special point')} yield the containment
\begin{align}
\lim_{\nu\to\infty} p\bigl((\psi_\gamma^\nu)^{-1}\bigr)(x_{i'}^\nu) = x_{\gamma\lambda_{i'}} \in \bR.
\end{align}
By \textsc{(special point')} and the stability condition for tree-pairs,
there exist $i'',j$ such that the equality $\lim_{\nu\to\infty} p\bigl((\psi_\alpha^\nu)^{-1}(z_{i''j})\bigr) = x_{\alpha,1}$ holds, so Step 1 and the last displayed containment imply the claim.
Indeed, write $\bigl((\psi_\alpha^\nu)^{-1}\circ\psi_\gamma^\nu\bigr)(z) = a^\nu z + b^\nu$; by Step 1, $\lim_{\nu\to\infty} a^\nu = 0$.
Now, for any $i'$, we have
\begin{align}
p\bigl((\psi_\alpha^\nu)^{-1}(z_{i'j})\bigr)
-
p\bigl((\psi_\alpha^\nu)^{-1}(z_{i''j})\bigr)
=
a^\nu\Bigl(p\bigl((\psi_\gamma^\nu)^{-1}(z_{i'j})\bigr)
-
p\bigl((\psi_\gamma^\nu)^{-1}(z_{i''j})\bigr)\Bigr)
\sr{\nu\to\infty}{\lra} 0,
\end{align}
hence $\lim_{\nu\to\infty} p\bigl((\psi_\alpha^\nu)^{-1}(z_{i'j})\bigr) = x_{\alpha,1}$.

A similar argument can be made in the case that $\pi(\alpha)$ is a leaf of $T_s$.
\end{proof}

\subsection{Gromov compactness for \texorpdfstring{$\ol{2\cM}_\bn$}{2Mn}}
\label{ss:compact}

This subsection is devoted to establishing the following result, which will later be used to show that the topology on $\ol{2\cM}_\bn$ is compact.

\begin{theorem} \label{thm:2M_gromov}
Any sequence $\bigl(2T^\nu, (\bx_\rho^\nu), (\bz_\alpha^\nu)\bigr) \subset \SWC_\bn$ has a Gromov-convergent subsequence.
\end{theorem}

\noindent The central idea of the proof already occurs when the witch curves in the sequence are smooth.
In this case, we prove this theorem inductively, on the total number of marked points.
The idea is that when we add a new marked point to a Gromov-convergent sequence of smooth witch curves, there are four possibilities, illustrated in the following figure and made formal in Lemma~\ref{lem:2M_add_pt}.

\begin{figure}[H]
\centering
\def\svgwidth{0.95\columnwidth}
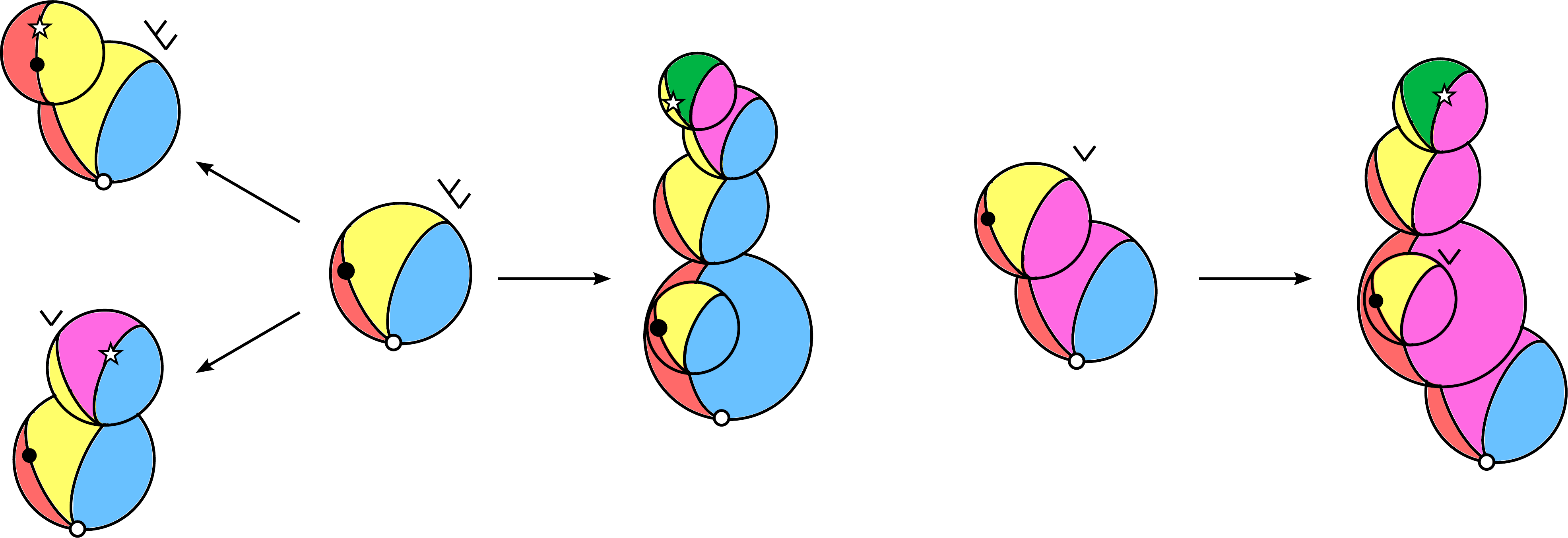
\caption{The sources of the arrows are two points in
$\SWC_{1000}$,
which we think of as limits of sequences of smooth witch curves.
The targets show examples of what the limit can become when we add an additional marked point to the original sequence of smooth witch curves.
Some of the seams in this figure are the results of several seams merging; these seams are indicated by small adjacent trees.
\label{fig:add_pt}}
\end{figure}

\begin{lemma} \label{lem:2M_add_pt}
Suppose that a sequence $(\bx^\nu, \bz^\nu) \subset \SWC_\bn$ of smooth stable witch curves Gromov-converges to $\bigl(2T, (\bx_\rho), (\bz_\alpha)\bigr)$ via $(\phi_\rho^\nu)$ and $(\psi_\alpha^\nu)$, and that $(\zeta^\nu \in \bR^2 \setminus \bz^\nu)$ is a sequence with the property that
\begin{equation} \label{eq:new_pt_converges}
\zeta_\alpha \coloneqq \lim_{\nu\to\infty} (\psi_\alpha^\nu)^{-1}(\zeta^\nu) \in \bR^2\cup \{\infty\}
\end{equation}
exists for every $\alpha \in V_\comp(T_b)$.  Then exactly one of the following conditions holds:
\begin{itemize}
\item[(1)] There exists a (unique) vertex $\alpha \in V_\comp(T_b)$ such that $\zeta_\alpha \in \bR^2 \setminus Z_\alpha^\spec$.

\item[(2a)] There exists a (unique) contiguous pair $\alpha \in V_\comp(T_b)$, $\mu_{ij}$ such that $\zeta_\alpha = z_{\alpha\mu_{ij}}$.

\item[(2b)] The root $\alpha_\root$ has $\zeta_{\alpha_\root} = \infty$.

\item[(3)] There exists a (unique) contiguous pair $\alpha, \beta \in V_\comp(T_b)$ such that $\zeta_\alpha = z_{\alpha\beta}$ and $\zeta_\beta = z_{\beta\alpha}$.
\end{itemize}
\end{lemma}

\begin{proof}
We imitate the proof of Lemma~5.3.4, \cite{ms:jh}.

\medskip
\noindent{\sc Step 1:} {\it We prove the implication
\begin{align} \label{eq:add_pt_helper}
\alpha,\beta \in V_\comp(T_b), \: \zeta_\alpha \neq z_{\alpha\beta} \implies \zeta_\beta = z_{\beta\alpha}.
\end{align}}

\noindent This follows from the {\sc (rescaling')} part of Lemma~\ref{lem:2M_gromov_conseq} and the convergence of $(\psi_\alpha^\nu)^{-1}(\zeta^\nu)$ to $\zeta_\alpha\neq z_{\alpha\beta}$:
\begin{equation}
\zeta_\beta = \lim_{\nu\to\infty} (\psi_\beta^\nu)^{-1}(\zeta^\nu) = \lim_{\nu\to\infty} \psi_{\beta\alpha}^\nu\bigl((\psi_\alpha^\nu)^{-1}(\zeta^\nu)\bigr) = z_{\beta\alpha}.
\end{equation}

\noindent {\sc Step 2:} {\it We prove the lemma.}
\medskip

\noindent We begin by proving that the four cases are mutually exclusive.
\begin{itemize}
\item Suppose that $\alpha, \beta$ satisfy the condition in (3), and fix $\gamma \in V_\comp(T_b) \setminus \{\alpha,\beta\}$.
If $\gamma$ lies in $(T_b)_{\alpha\beta}$, then the inequality $\zeta_\beta = z_{\beta\alpha} \neq z_{\beta\gamma}$ and Step 1 imply $\zeta_\gamma = z_{\gamma\beta}$, so none of (1), (2a), and (2b) hold.
Otherwise, the inequality $\zeta_\alpha = z_{\alpha\beta}\neq z_{\alpha\gamma}$ and Step 1 imply $\zeta_\gamma = z_{\gamma\alpha}$, so none of (1), (2a), and (2b) hold.

\item Suppose that (2b) holds.
Step 1 implies that every $\alpha \in V_\comp(T_b)$ has $\zeta_\alpha = \infty$, so neither (1) nor (2a) holds.

\item Suppose that $\alpha, \mu_{ij}$ satisfy (2a).
Step 1 implies that every $\beta \in V_\comp(T_b) \setminus \{\alpha\}$ has $\zeta_\beta = z_{\beta\alpha}$, so (1) does not hold.
\end{itemize}

Next, we prove uniqueness in (1), (2a), and (3).
In (1) and (2a), this is an immediate consequence of Step 1.
To prove uniqueness in (3), suppose for a contradiction that $\{\alpha,\beta\}$ and $\{\alpha',\beta'\}$ are distinct pairs satisfying (3).
Switching $\alpha$ and $\beta$ if necessary, we may assume that the paths from $\alpha'$ resp.\ $\beta'$ to $\alpha$ pass through $\beta$.
Similarly, we may assume that the paths from $\alpha$ resp.\ $\beta$ to $\beta'$ pass through $\alpha'$.
The inequality $\zeta_\beta = z_{\beta\alpha} \neq z_{\beta\alpha'}$ and Step 1 imply $\zeta_{\alpha'} = z_{\alpha'\beta}$.
This, together with the inequality $z_{\alpha'\beta} \neq z_{\alpha'\beta'}$, imply $\zeta_{\alpha'} = z_{\alpha'\beta} \neq z_{\alpha'\beta'}$, in contradiction with the assumption.

Finally, we show that at least one of these cases holds.
Suppose that (1), (2a), and (2b) do not hold; we must show that (3) holds.
The assumption implies that for every $\alpha \in V_\comp(T_b)$, there exists a
(unique)
contiguous $\beta \in V_\comp(T_b)$ with $\zeta_\alpha = z_{\alpha\beta}$.
Define a
(possibly self-crossing)
path like so: First, choose any $\alpha_1 \in V_\comp(T_b)$ and a contiguous vertex $\alpha_2 \in V_\comp(T_b)$ with $\zeta_{\alpha_1} = z_{\alpha_1\alpha_2}$.
Inductively continue this path by defining $\alpha_{k+1} \in V_\comp(T_b)$ to be the vertex in $V_\comp(T_b)$ contiguous to $\alpha_k$ satisfying $\zeta_{\alpha_k} = z_{\alpha_k\alpha_{k+1}}$.
The quotient of $T_b$ obtained by identifying each element of $V_\comp(T_b)$ with its incoming neighbors is again a tree, and $(\alpha_1,\alpha_2,\ldots)$ is an infinite path in this quotient, so there must exist $k$ such that $\alpha_{k+2} = \alpha_k$.
Then $\zeta_{\alpha_k} = z_{\alpha_k\alpha_{k+1}}$ and $\zeta_{\alpha_{k+1}} = z_{\alpha_{k+1}\alpha_k}$, so $\alpha_k,\alpha_{k+1}$ satisfy the condition in (3).
\end{proof}

\begin{proof}[Proof of Thm.~\ref{thm:2M_gromov}]
\noindent {\sc Step 1:} {\it For any $r\geq 2$ and $\bn \in \bZ_{\geq0}^r$ with $|\bn| = 1$, there is a bijection $\SDT_r \to \SWC_\bn$
that
identifies Gromov-convergent sequences with Gromov-convergent sequences.}
\medskip

Fix $\bn$ as above, where the only nonzero entry is $n_{i_0} = 1$.
We begin by identifying stable RRTs with $r$ leaves with stable tree-pairs of type $\bn$.
\begin{itemize}
\item Given a stable RRT $T$ with $r$ leaves, we define a stable tree-pair $2T$ of type $\bn$ like so: set $T_s \coloneqq T$.
Define $T_b$ by first setting $T'$ to consist of all vertices in the path $[\rho_\root^T,\lambda_{i_0}[$ and all incoming neighbors of these vertices; now, define $T_b$ to be the result of inserting a dashed edge at $\lambda_{i_0}$, and at every interior vertex of $T'$ except the root.
Here is an illustration of this process, in which a stable RRT with 5 leaves is sent to a
stable
tree-pair of type $(0,0,0,1,0)$:

\vspace{-0.5em}
\begin{figure}[H]
\centering
\def\svgwidth{0.4\columnwidth}
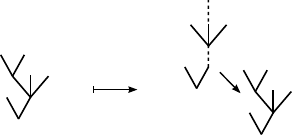
\end{figure}

\item Given a stable tree-pair $2T$ of type $\bn$, send it to its seam tree $T_s$.
\end{itemize}

\noindent The fact that these maps are inverses follows from the stability condition on stable tree-pairs.

We now enhance this bijection to an identification of $\SDT_r$ with $\SWC_\bn$.
Fix $\bigl(T, (\bx_\rho)\bigr) \in \SDT_r$.
Define $2T$ as above.
Define $\bigl(2T, (\bx_\rho), (\bz_\alpha) \bigr) \in \SWC_\bn$ like so: for $\alpha \in V_\comp(T_b)$, choose $i_0$ with the property that the $i_0$-th incoming neighbor $\beta$ of $\alpha$ has $\incom(\beta) \neq \emptyset$, and set $\bz_\alpha \coloneqq \bigl((x_{\pi(\alpha),i_0},0)\bigr)$.
It is straightforward to check that this indeed defines a bijection, and that it identifies Gromov-convergent sequences in $\SDT_r$ with Gromov-convergent sequences in $\SWC_\bn$.

\medskip
\noindent {\sc Step 2:} {\it If $(\bx^\nu,\bz^\nu) \subset \SWC_\bn$ is a sequence of smooth stable witch curves, then it has a Gromov-convergent subsequence.}
\medskip

We establish this claim by induction on $|\bn|$.
The base case $\bn = (2)$ follows from the fact that any two elements of $\SWC_{(2)}$ are isomorphic, while the base case $r\geq 2$, $|\bn| = 1$ follows from Step 1.

Next, say that the claim has been proven up to,
but not including,
some $|\bn| = a \geq 1$.
(In the $r=1$ case, start with $a \geq 2$ instead.)
Fix a sequence $(\bx^\nu,\bz^\nu) \in \SWC_\bn$, with $|\bn| = a$.
Without loss of generality, we may choose $i_0$ such that the inequality $y_{i_0,n_{i_0}} \geq y_{ij}$ holds for all $i,j$.
Define $\wt\bn \coloneqq (n_1,\ldots,n_{i_0-1},n_{i_0}-1,n_{i_0+1},\ldots,n_r)$ and $\wt\bz^\nu \coloneqq (z_{ij}^\nu \:|\: (i,j) \neq (i_0,n_{i_0}))$.
By the inductive hypothesis, we may assume that $(\bx^\nu,\wt\bz^\nu) \subset \SWC_{\wt\bn}$ Gromov-converges to some $\bigl(2T, (\bx_\rho), (\bz_\alpha)\bigr)$.
Set $\zeta^\nu \coloneqq z_{i_0,n_{i_0}}$.
Passing to a subsequence, we may assume that for every $\alpha \in V_\comp(T_b)$, the limit $\zeta_\alpha \coloneqq \lim_{\nu\to\infty} (\psi_\alpha^\nu)^{-1}(\zeta^\nu) \in \bR^2 \cup \{\infty\}$ exists.
(Indeed, enumerate $V_\comp(T_b)$ as $\alpha_1,\ldots,\alpha_n$.
Since $\bR^2 \cup \{\infty\}$ is compact, we may pass to a subsequence so that the limit defining $\zeta_{\alpha_1}$ exists.
Next, we pass to a further subsequence so that the limit defining $\zeta_{\alpha_2}$ exists, and so forth.)
We may now apply Lemma~\ref{lem:2M_add_pt}; we divide the rest of the proof of this step into cases, depending on which case of Lemma~\ref{lem:2M_add_pt} holds.
\begin{itemize}
\item[(1)] Fix $\alpha \in V_\comp(T_b)$ with the property $\zeta_\alpha \in \bR^2 \setminus Z_\alpha^\spec$.
We begin by defining a
stable
tree-pair $2T^\new$ of type $\bn$.
Set $T_s^\new \coloneqq T_s$.
Enlarge $T_b$ to $T_b^\new$ like so: define $T'$ to be the subtree of $T_s$ consisting of the path $]\pi(\alpha),\lambda_{i_0}[$, together with all incoming neighbors of all vertices in this path.
Insert a dashed edge at every interior vertex in $T'$, including at its root.
Add an incoming dashed edge to the vertex in $T'$ corresponding to $\lambda_{i_0}$.
Finally, graft this tree into $T_b$ by identifying its root with the vertex $\beta \in \incom(\alpha)$ with the property that the path from $\pi(\alpha)$ to $\lambda_{i_0}$ passes through $\pi(\beta)$; declare that the element we have just added to $\incom(\beta)$ is maximal in $\incom(\beta)$.

Next, we define a collection of reparametrizations $(\chi_{\beta_\rho}^\nu)$, where $\beta_\rho$ denotes the vertex we added to $T_b^\new$ corresponding to $\rho \in ]f(\alpha), \lambda_{i_0}[$.
We characterize $\chi_{\beta_\rho}^\nu$ by the following equations:
\begin{align}
p(\chi_{\beta_\rho}^\nu) = \phi_\rho^\nu, \qquad (\chi_{\beta_\rho}^\nu)^{-1}(x_{i_0}^\nu,y_{i_0,n_{i_0}}^\nu) = \bigl((\phi_\rho^\nu)^{-1}(x_{i_0}),0\bigr).
\end{align}
The sequence $(\bx^\nu,\bz^\nu)$ converges to $\bigl(2T^\new,(\bx_\rho),(\bz_\alpha)\bigr)$ via $(\phi_\rho^\nu)$ and $(\psi_\alpha^\nu) \cup (\chi_{\beta_\rho}^\nu)$.

\item[(2a)] Fix a contiguous pair $\alpha \in V_\comp(T_b), \mu_{ij}$ with $\zeta_\alpha = z_{\alpha\mu_{ij}}$.
We begin by defining $2T^\new$:
\begin{align}
T_s^\new \coloneqq T_s, \qquad T_b^\new \coloneqq T_b \cup \{\alpha', \mu_{i_0,n_{i_0}-1}^{T_b^\new}, \mu_{i_0,n_{i_0}}^{T_b^\new}\},
\end{align}
where $\alpha'$ is a new seam vertex attached via an outgoing solid edge to $\beta\coloneqq \mu_{i_0,n_{i_0}-1}^{T_b}$ (which is converted from an element of $(\mu_{ij})_{i,j}$ to an element of $V_\comp(T_b^\new)$), and where $\mu_{i_0,n_{i_0}-1}^{T_b^\new},\mu_{i_0,n_{i_0}}^{T_b^\new}$ are the incoming neighbors of $\alpha'$.

Next, we define $\chi_\beta^\nu$ by these equations:
\begin{align}
(\chi_\beta^\nu)^{-1}(z_{i_0,n_{i_0}-1}^\nu) = (0,0),
\qquad
(\chi_\beta^\nu)^{-1}(z_{i_0,n_{i_0}}^\nu) = (0,1).
\end{align}
Then $(\bx^\nu,\bz^\nu)$ converges to $\bigl(2T^\new,(\bx_\rho),(\bz_\alpha)\bigr)$ via $(\phi_\rho^\nu)$ and $(\psi_\alpha^\nu) \cup (\chi_\beta^\nu)$.

\item[(2b)] Suppose $\zeta_{\alpha_\root} = \infty$.
Set $T_s^\new \coloneqq T_s$, and define $T_b^\new$ like so: define $T'$ to be the subtree of $T_s$ consisting of $[\rho_\root^{T_s},\lambda_{i_0}[$ and all incoming neighbors of these vertices.
Insert dashed edges at all the interior vertices of $T'$ besides the root, and at the vertex of $T'$ corresponding to $\lambda_{i_0}$.
Complete the construction of $T_b^\new$ by introducing $\alpha' \eqqcolon \alpha_\root^{T_b^\new}$, connect $\alpha'$ by an incoming solid edge to a new seam vertex $\alpha''$, and connect both $\alpha_\root^{T_b}$ and the root of $T'$ to $\alpha''$ by dashed edges.

Next, we need to define reparametrizations $\chi_{\beta_\rho}^\nu, \chi_{\alpha'}^\nu$, where $\beta_\rho$ is the vertex in $T'$ corresponding to $\rho \in [\rho_\root^{T_s},\lambda_{i_0}^{T_s}[$.
The definition of $\chi_{\beta_\rho}^\nu$ is similar to the construction of the reparametrizations in (1).
To define $\chi_{\alpha'}^\nu$, choose $i_1 \in [1,r]$ and $j_1 \in [1,\wt n_{i_1}]$ and characterize $\chi_{\alpha'}^\nu$ by the equations
\begin{align}
(\chi_{\alpha'}^\nu)^{-1}(z_{i_0,n_{i_0}}^\nu) = (*,1), \qquad (\chi_{\alpha'}^\nu)^{-1}(z_{i_1j_1}) = (0,0).
\end{align}

\item[(3)] This case is similar to (2b), so we do not include all the details.
As illustrated in Fig.~\ref{fig:add_pt}, to obtain the new stable witch curve we introduce a new component between the spheres corresponding to $\alpha$ and $\beta$, and possibly attach a further bubble tree to this new component.
More precisely, we first enlarge $V_\comp(T_b)$ and $V_\seam(T_b)$ by adding a component vertex $\gamma$ to $V_\comp(T_b)$ between $\alpha$ and $\beta$, adding a seam vertex $\delta$ to $V_\seam(T_b)$, and setting $\incom(\gamma) \coloneqq (\delta)$.
To further enlarge $T_b$, set $T'$ to be the subtree of $T_s$ consisting of $]\pi(\alpha),\lambda_{i_0}[$ and all incoming neighbors of these vertices.
As in (2b), insert dashed edges at all the interior vertices of $T'$ besides the root, and at the vertex of $T'$ corresponding to $\lambda_{i_0}$.
Finally, attach the root of $T'$ to $\gamma$ via a dashed edge.

The construction of the reparametrizations corresponding to the new vertices is somewhat different than the constructions appearing in (2b); we therefore concentrate on this detail.
Specifically, we show how to construct the reparametrizations $\chi^\nu$ corresponding to $\gamma$.

Suppose that $\alpha,\beta\in V_\comp(T_b)$ are contiguous and have the properties $\zeta_\alpha = z_{\alpha\beta}$ and $\zeta_\beta = z_{\beta\alpha}$.
Assume w.l.o.g.\ that $\beta$ is further from the root than $\alpha$ is, which implies $z_{\beta\alpha} = \infty$.
We must construct a sequence $(\chi^\nu) \subset G_1$ satisfying these conditions:
\begin{itemize}
\item[($\chi1$)] $(\chi^\nu)^{-1}(z_{i_0,n_{i_0}}^\nu) = (0,1)$ for every $\nu$.

\item[($\chi2$)] $(\psi_\alpha^\nu)^{-1} \circ \chi^\nu$ converges to 0 u.c.s.\ away from $\infty$ and $(\psi_\beta^\nu)^{-1} \circ \chi^\nu$ converges to $\infty$ u.c.s.\ away from $z_{\alpha\beta}$.
\end{itemize}
To do so, we first note that we may assume $z_{\alpha\beta} = 0$: otherwise, set $\xi \in G_2$ to be translation by $z_{\alpha\beta}$ and replace $\psi_\alpha^\nu$ by $\psi_\alpha^\nu\circ \xi$ and\ $\bz_\alpha^\nu$  by\ $\xi^{-1}(\bz_\alpha^\nu)$.
The sequence $(\psi_{\alpha\beta}^\nu) = \bigl((\psi_\alpha^\nu)^{-1}\circ\psi_\beta^\nu\bigr)$ converges to 0 u.c.s.\ away from $\infty$, so if we write $\psi_{\alpha\beta}^\nu(z) = a^\nu z + b^\nu$, the sequences $(a^\nu) \subset \bR_{>0}$ and $(b^\nu) \subset \bR^2$ both converge to 0.
Set $w^\nu \coloneqq (\psi_\beta^\nu)^{-1}(z_{i_0,n_{i_0}}^\nu)$; we then have
\begin{equation} \label{eq:2M_compact_3_limits}
\lim_{\nu\to\infty} w^\nu = \zeta_\beta = z_{\beta\alpha} = \infty, \qquad \lim_{\nu\to\infty} a^\nu w^\nu = \lim_{\nu\to\infty} \psi_{\alpha\beta}^\nu(w^\nu) = \zeta_\alpha = z_{\alpha\beta} = 0.
\end{equation}
{\sc(Restriction)} and Lemma~\ref{lem:V1_x-limits} imply that $\Re(w^\nu)$ is bounded; just as we assumed $z_{\alpha\beta} = 0$, we may therefore assume $w^\nu = (0,c^\nu)$ for $c^\nu \in \bR$.
Moreover, the inequality $y_{i_0,n_{i_0}} \geq y_{i'j'}$ for all $i',j'$ implies that $c^\nu$ is eventually positive.
The functions $\xi^\nu$ defined by $\xi^\nu(z) \coloneqq (z - b^\nu)/(a^\nu c^\nu)$ therefore lie in $G_2$ and satisfy these conditions:
\begin{itemize}
\item[($\xi1$)] $\xi^\nu$ converges to $\infty$ u.c.s.\ away from 0.

\item[($\xi2$)] $\xi^\nu \circ \psi_{\alpha\beta}^\nu = (z \mapsto z/c^\nu)$ converges to 0 u.c.s.\ away from $\infty$.

\item[($\xi3$)] The following identity holds for every $\nu$:
\begin{equation}
\bigl(\xi^\nu\circ(\psi_\alpha^\nu)^{-1}\bigr)(z_{i_0,n_{i_0}}^\nu) = (\xi^\nu\circ \psi_{\alpha\beta}^\nu)\bigl((\psi_\beta^\nu)^{-1}(z_{i_0,n_{i_0}}^\nu)\bigr) = (\xi^\nu\circ\psi_{\alpha\beta}^\nu)(w^\nu) = (0,1).
\end{equation}
\end{itemize}

\noindent Now set
\begin{equation}
\chi^\nu \coloneqq \psi_\alpha^\nu \circ (\xi^\nu)^{-1}.	
\end{equation}
Then $(\xi3)$ implies $(\chi^\nu)^{-1}(z_{i_0,n_{i_0}}^\nu)  = (0,1)$, which establishes $(\chi1)$.
$(\xi1)$ and $(\xi2)$ imply $(\chi2)$, so we have constructed a suitable rescaling sequence $(\chi^\nu)$.
\end{itemize}
\end{proof}

\subsection{Limits in \texorpdfstring{$\ol{2\cM}_\bn$}{2Mn} are unique}
\label{ss:unique_limits}

\begin{lemma} \label{lem:2M_insert_edge}
Suppose that $(\bx^\nu,\bz^\nu) \subset \SWC_\bn$ is a sequence of smooth stable witch curves that Gromov-converges to $\bigl(2T,(\bx_\rho^\nu),(\bz_\alpha^\nu)\bigr)$ via $(\phi_\rho^\nu)$ and $(\psi_\alpha^\nu)$, that $\alpha, \beta \in V_\comp(T_b)$ are contiguous vertices, and that $(\chi^\nu) \subset G_2$ is a sequence with the property
\begin{gather} \label{eq:2M_insert_edge}
(\psi_\alpha^\nu)^{-1}\circ \chi^\nu \to z_{\alpha\beta} \:\:\text{u.c.s.\ away from}\:\: w_1,
\qquad
(\psi_\beta^\nu)^{-1}\circ \chi^\nu \to z_{\beta\alpha} \:\:\text{u.c.s.\ away from}\:\: w_2
\end{gather}
for some $w_1,w_2 \in \bR^2\cup\{\infty\}$.
If $\alpha$ is closer to the root than $\beta$, then $w_1=\infty$; otherwise, $w_2 = \infty$.
If $\mu_{ij}$ lies in $(T_b)_{\alpha\beta}$, then $(\chi^\nu)^{-1}(z_{ij}^\nu)$ converges to $w_2$; otherwise, $(\chi^\nu)^{-1}(z_{ij}^\nu)$ converges to $w_1$.
If $x_{\beta\lambda_i} \neq x_{\beta\alpha}$, then $p(\chi^\nu)^{-1}(x_i^\nu)$ converges to $p(w_2)$; otherwise, $p(\chi^\nu)^{-1}(x_i^\nu)$ converges to $p(w_1)$.
\end{lemma}

\begin{proof}
To prove the first claim, it suffices by symmetry to consider the case that $\alpha$ is closer to the root than $\beta$.
In this case we have $z_{\alpha\beta} \in \bR^2$, so the first equation in \eqref{eq:2M_insert_edge} and the equality $\bigl((\psi_\alpha^\nu)^{-1}\circ\chi^\nu\bigr)(\infty) = \infty$ imply $w_1=\infty$.

To prove the second claim, it suffices by symmetry to consider the case that $\mu_{ij}$ lies in $(T_b)_{\alpha\beta}$.
Suppose that $(\chi^\nu)^{-1}(z_{ij}^\nu)$ does not converge to $w_2$.
Passing to a subsequence, we may assume that there exists a compact set $K \not\,\ni w_2$ with $(\chi^\nu)^{-1}(z_{ij}^\nu) \in K$ for all $\nu$.
By hypothesis, we have $(\psi_\beta^\nu)^{-1}(z_{ij}^\nu) = \bigl((\psi_\beta^\nu)^{-1}\circ\chi^\nu\bigr)\bigl((\chi^\nu)^{-1}(z_{ij}^\nu)\bigr) \to z_{\beta\alpha}$.
On the other hand, {\sc (special point')} implies that $(\psi_\beta^\nu)^{-1}(z_{ij}^\nu)$ converges to $z_{\beta\mu_{ij}}$, hence $z_{\beta\alpha} = z_{\beta\mu_{ij}}$.
This contradicts the
hypothesis that $\mu_{ij}$ lies in $(T_b)_{\alpha\beta}$.

A similar argument proves the third claim.
\end{proof}

\begin{lemma} \label{lem:2M_commensurate}
Suppose that $(\bx^\nu,\bz^\nu) \subset \SWC_\bn$ and $(\psi_\alpha^\nu) \subset G_2$ are as in Lemma~\ref{lem:2M_insert_edge}, and suppose that $(\chi^\nu) \subset G_2$ is a sequence of reparametrizations with the following properties:
\begin{itemize}
\item[(a)] For every $i, j$ the limits $\zeta_i \coloneqq \lim_{\nu\to\infty} p(\chi^\nu)^{-1}(x_i^\nu)$, $\xi_{ij} \coloneqq \lim_{\nu\to\infty} (\chi^\nu)^{-1}(z_{ij}^\nu)$ exist.

\item[(b)] Define $Y_s \coloneqq (\zeta_i)_i \cup \{\infty\}$ and $Y_b \coloneqq (\xi_{ij})_{i,j} \cup \{\infty\}$.
Either $\# Y_b \geq 3$, or $\# Y_b = 2$ and $\#Y_s \geq 3$.
\end{itemize}
Then there exists $\alpha \in V_\comp(T_b)$ such that $\bigl((\psi_\alpha^\nu)^{-1}\circ\chi^\nu\bigr)$ has a subsequence
that
converges uniformly to an element of $G_2$.
\end{lemma}

\begin{proof}
\noindent {\sc Step 1:} {\it If $(\tau^\nu) \subset G_2$ has no convergent subsequence, then it has a subsequence converging to $w$ u.c.s.\ away from $w'$ for some $w,w' \in \bR^2\cup\{\infty\}$.}
\medskip

\noindent Write $\tau^\nu(z) = a^\nu z + b^\nu$.
After passing to a subsequence, we may assume that the limits
\begin{equation}
\lim_{\nu\to\infty} a^\nu \eqqcolon a^\infty \in \bR_{\geq0}\cup\{\infty\}, \quad \lim_{\nu\to\infty} b^\nu \eqqcolon b^\infty \in \bR^2\cup\{\infty\}
\end{equation}
exist.
It suffices to prove the claim for either $(\tau^\nu)$ or $\bigl((\tau^\nu)^{-1}\bigr)$; replacing $\tau^\nu$ by $(\tau^\nu)^{-1}$ if necessary, we may assume $a^\infty \in \bR_{\geq0}$.
By hypothesis, it cannot be that the containments $a^\infty \in \bR_{>0}$, $b^\infty \in \bR^2$ both hold.
If $b^\infty = \infty$, then $\tau^\nu$ converges to $\infty$ u.c.s.\ away from $\infty$.
If $a^\infty = 0$ and $b^\infty \in \bR^2$, then $\tau^\nu$ converges to $b^\infty$ u.c.s.\ away from $\infty$.

\medskip
\noindent {\sc Step 2:} {\it If $\tau^\nu \coloneqq (\psi_\alpha^\nu)^{-1}\circ \chi^\nu$ has no uniformly-convergent subsequence, then after passing to a subsequence, $\tau^\nu$ converges to $w$ u.c.s.\ away from $w'$ for some $w \in Z_\alpha^\node$ and $w' \in Y_b$.}
\medskip


\noindent
By Step 1, we may pass to a subsequence such that $\tau^\nu$ converges to $w$ u.c.s.\ away from $w'$ for some $w, w' \in \bR^2\cup\{\infty\}$; it remains to show $w \in Z_\alpha^\node$, $w' \in Y_b$.
Suppose for a contradiction that $w$ does not lie in $Z_\alpha^\node$.
Then at most one $\mu \in (\mu_{ij}) \cup \{\mu_\infty^{T_b}\}$ satisfies $z_{\alpha\mu} = w$.
Assume from now on that there is either (i) no such $\mu$, or (ii) the only such $\mu$ is $\mu = \mu_\infty^{T_b}$; the case $w = z_{\alpha\mu_{ij}}$ is similar.
The reparametrizations $(\tau^\nu)^{-1}$ converge to $w'$ u.c.s.\ away from $w$ and $(\psi_\alpha^\nu)^{-1}(z_{ij}^\nu)$ converges to $z_{\alpha\mu_{ij}} \neq w$ by
{\sc (special point')},
so for every $i,j$ we have
\begin{equation}
\xi_{ij} = \lim_{\nu\to\infty} (\chi^\nu)^{-1}(z_{ij}^\nu) = \lim_{\nu\to\infty} (\tau^\nu)^{-1}\bigl((\psi_\alpha^\nu)^{-1}(z_{ij}^\nu)\bigr) = w'.
\end{equation}
This and the inequality $\# Y_b \geq 2$ implies that $\# Y_b = 2$.
Moreover, (ii) must hold rather than (i): $\#Y_b \geq 2$ implies $w' \neq \infty$, and it is only possible for a sequence in $G_2$ to converge to $w' \neq \infty$ away from $w \in \bR^2\cup\{\infty\}$ if $w$ is equal to $\infty$.

Next, note that the facts $w \not\in Z_\alpha^\node$ and $w = \infty$ imply $\alpha = \alpha_\root^{T_b}$, hence
\begin{align}
\label{eq:new_bubble_psi_alpha_limits}
\lim_{\nu\to\infty} p(\psi_\alpha^\nu)^{-1}(x_i^\nu) = x_{\alpha_\root\lambda_i} \in \bR \quad\forall\: i.
\end{align}
The inequality $w' \neq \infty$ implies that $p(\tau^\nu)^{-1}$ converges to $p(w')$ u.c.s.\ away from $\infty$.
\eqref{eq:new_bubble_psi_alpha_limits} now implies that for any $i$, we have
\begin{align}
\zeta_i = \lim_{\nu\to\infty} p(\chi^\nu)^{-1}(x_i^\nu) = \lim_{\nu\to\infty} p(\tau^\nu)^{-1}\bigl(p(\psi_\alpha^\nu)^{-1}(x_i^\nu)\bigr) = p(w').
\end{align}
Therefore $\# Y_s \leq 2$.
Together with the equality $\# Y_b = 2$, we have derived a contradiction.

A similar argument shows $w' \in Y_b$.

\medskip
\noindent {\sc Step 3:} {\it If the conclusion of Lemma~\ref{lem:2M_commensurate} does not hold, then there is a contradiction.}
\medskip

\noindent Suppose that no $\alpha \in V_\comp(T_b)$ has the property that a subsequence of $\bigl((\psi_\alpha^\nu)^{-1}\circ\chi^\nu\bigr)$ converges uniformly; we will construct a non-self-intersecting infinite sequence $(\alpha_1,\alpha_2,\ldots)$ in $V_\comp(T_b)$ with every consecutive pair $\alpha_i, \alpha_{i+1}$ contiguous, a contradiction.
We choose $\alpha_1$ to be any element of $V_\comp(T_b)$.
By Step 2, we may pass to a subsequence such that $(\psi_{\alpha_1}^\nu)^{-1}\circ\chi^\nu$ converges to $w_1 \in Z^\node_{\alpha_1}$ u.c.s.\ away from $w_1' \in Y_b$; define $\alpha_2 \in V_\comp(T_b)$ to be the vertex contiguous to $\alpha_1$ with $w_2 = z_{\alpha_1\alpha_2}$.
Inductively defining our sequence in this fashion, we obtain $(\alpha_1,\alpha_2,\ldots)$ with the property that $(\psi_{\alpha_i}^\nu)^{-1}\circ\chi^\nu$ converges to $z_{\alpha_i\alpha_{i+2}}$ u.c.s.\ away from $w_i'$.
This path does not intersect itself: Indeed, assume that $\alpha_i = \alpha_{i+2}$ for some $i$.
Then Lemma~\ref{lem:2M_insert_edge} with
$\alpha \coloneqq \alpha_i$, $\beta \coloneqq \alpha_{i+1}$
implies $\#Y_b \leq 2$ and $\#Y_s \leq 2$, a contradiction.
We have therefore constructed an infinite sequence in $V_\comp(T_b)$ with each consecutive pair contiguous, a contradiction.
\end{proof}

\begin{theorem}
\label{thm:2Mn_unique_limits}
Suppose that $\bigl(2T^\nu, (\bx_\rho^\nu), (\bz_\alpha^\nu)\bigr) \subset \SWC_\bn$ Gromov-converges to two stable witch curves $\bigl(2T,(\bx_\rho),(\bz_\alpha)\bigr)$ and $\bigl(\wt{2T},(\wt \bx_{\wt\rho}),(\wt\bz_{\wt\alpha})\bigr)$.
Then $\bigl(2T,(\bx_\rho),(\bz_\alpha)\bigr)$ and $\bigl(\wt{2T},(\wt \bx_{\wt\alpha}),(\wt\bz_{\wt\alpha})\bigr)$ are isomorphic.
\end{theorem}

\begin{proof}
\medskip
\noindent {\sc Step 1:} {\it If $(\bx^\nu,\bz^\nu) \subset \SWC_\bn$ is a sequence of smooth stable disk trees Gromov-converging to $\bigl(2T,(\bx_\rho),(\bz_\alpha)\bigr)$ and $\bigl(\wt{2T},(\wt \bx_{\wt\rho}),(\wt\bz_{\wt\alpha})\bigr)$, then $\bigl(2T,(\bx_\rho),(\bz_\alpha)\bigr)$ and $\bigl(\wt{2T},(\wt \bx_{\wt\rho}),(\wt\bz_{\wt\alpha})\bigr)$ are isomorphic.}
\medskip

\medskip
\noindent {\sc Step 1a:} {\it The stable disk trees $\bigl(T_s,(\bx_\rho)\bigr)$ and $\bigl(\wt T_s,(\wt\bx_{\wt\rho})\bigr)$ are isomorphic.}
\medskip

\noindent This is a consequence of the Hausdorffness of $\ol\cM_r$.
We may therefore assume $\bigl(T_s,(\bx_\rho)\bigr) = \bigl(\wt T_s,(\wt\bx_{\wt\rho})\bigr)$.

\medskip


\noindent {\sc Step 1b:} {\it After passing to a subsequence, there is a unique bijection $g\colon V_\comp(T_b) \to V_\comp(\wt T_b)$ such that the uniform limits
\begin{equation} \label{eq:2M_haus_comm}
\chi_\alpha \coloneqq \lim_{\nu\to\infty} (\psi_\alpha^\nu)^{-1} \circ \wt\psi_{g(\alpha)}^\nu \in G_2
\end{equation}
exist.}
\medskip

\noindent Fix $\wt\alpha \in V_\comp(\wt T_b)$.
Applying Lemma~\ref{lem:2M_commensurate} with $\chi^\nu \coloneqq \wt\psi_{\wt\alpha}^\nu$, we see that there exists $\alpha \in V_\comp(T_b)$ such that a subsequence of $(\psi_\alpha^\nu)^{-1}\circ \wt\psi_{\wt\alpha}^\nu$ converges uniformly to an element of $G_2$.
Moreover, $\alpha$ is uniquely determined: indeed, for any other $\beta \in V_\comp(T_b)$, {\sc (rescaling')} implies that
\begin{equation}
(\psi_\beta^\nu)^{-1}\circ\wt\psi_{\wt\alpha}^\nu = \bigl((\psi_\beta^\nu)^{-1}\circ\psi_\alpha^\nu\bigr) \circ \bigl((\psi_\alpha^\nu)^{-1}\circ\wt\psi_{\wt\alpha}^\nu\bigr)
\end{equation}
converges to $z_{\beta\alpha}$ u.c.s.\ away from a single point.
By applying this argument at every interior vertex of $\wt T_b$, we obtain a uniquely-determined function $h\colon V_\comp(\wt T_b) \to V_\comp(T_b)$ and a subsequence of our original data such that the uniform limit
\begin{align}
\lim_{\nu\to\infty} \bigl(\psi_{h(\wt\alpha)}^\nu\bigr)^{-1}\circ\wt\psi_{\wt\alpha}^\nu \in G_2
\end{align}
exists.
Applying this reasoning with $T_b$ and $\wt T_b$ interchanged shows that $h$ is invertible; set $g \coloneqq h^{-1}$.

\medskip
\noindent {\sc Step 1c:} {\it The reparametrizations $\chi_\alpha$ satisfy
\begin{align}
\wt z_{g(\alpha)\wt\mu_{ij}} = \chi_\alpha^{-1}(z_{\alpha\mu_{ij}}),
&\qquad
\wt z_{g(\alpha)g(\beta)} = \chi_\alpha^{-1}(z_{\alpha\beta}), \\
\wt x_{g(\alpha)\wt\lambda_i} = p(\chi_\alpha)^{-1}(x_{\alpha\lambda_i}),
&\qquad
\wt x_{g(\alpha)g(\beta)} = p(\chi_\alpha)^{-1}(x_{\alpha\beta}). \nonumber
\end{align}}

\noindent The first equation follows from \eqref{eq:2M_haus_comm} and {\sc (special point')}:
\begin{align}
\chi_\alpha^{-1}(z_{\alpha\mu_{ij}})
=
\chi_\alpha^{-1}\Bigl(\lim_{\nu\to\infty} (\psi_\alpha^\nu)^{-1}(z_{ij}^\nu)\Bigr)
=
\lim_{\nu\to\infty} \bigl(\wt\psi_{g(\alpha)}^\nu\bigr)^{-1}(z_{ij}^\nu)
=
\wt z_{g(\alpha)\wt\mu_{ij}}.
\end{align}
A similar deduction proves the third equation.
The second follows from {\sc (rescaling')}, the convergence of $\psi_{\alpha\beta}^\nu$ to $z_{\alpha\beta}$ u.c.s.\ away from $z_{\beta\alpha}$, and the convergence of $\wt\psi_{g(\alpha)g(\beta)}^\nu$ to $\wt z_{g(\alpha)g(\beta)}$ u.c.s.\ away from $\wt z_{g(\beta)g(\alpha)}$.
Indeed, choosing $z \in \bR^2\setminus\{z_{\beta\alpha}, \chi_\beta(\wt z_{g(\beta)g(\alpha)})\}$, we have:
\begin{align}
\chi_\alpha^{-1}(z_{\alpha\beta})
=
\chi_\alpha^{-1}\Bigl(\lim_{\nu\to\infty} \psi_{\alpha\beta}^\nu(z)\Bigr)
=
\lim_{\nu\to\infty} \Bigl(\bigl(\wt\psi_{g(\alpha)}^\nu\bigr)^{-1}\circ \psi_\beta^\nu\Bigr)(z)
&=
\lim_{\nu\to\infty} \bigl(\wt\psi_{g(\alpha)g(\beta)}^\nu\bigr)^{-1}\bigl(\chi_\beta^{-1}(z)\bigr) \\
&= \wt z_{g(\alpha)g(\beta)}. \nonumber
\end{align}
Finally, the fourth equation follows from applying $p$ to the second equation.

\medskip
\noindent {\sc Step 1d:} {\it We extend $g$ to a bijection $V(T_b) \to V(\wt T_b)$.}
\medskip

\noindent We showed in Step 1b that $g\colon V_\comp(T_b) \to V_\comp(\wt T_b)$ is a bijection.
We now extend $g$ to a bijection between $V(T_b)$ and $V(\wt T_b)$.
First, set $g(\mu_{ij}^{T_b}) \coloneqq \mu_{ij}^{\wt T_b}$.
Next, suppose that $\alpha$ is an element of $V_\comp^{\geq 2}(T_b)$.
By {\sc(restriction)}, $\incom(\alpha)$ is in bijection with the limit set $\bigl(\lim_{\nu\to\infty} p\bigl((\psi_\alpha^\nu)^{-1}\bigr)(x_i^\nu)\:|\: 1\leq i \leq r\bigr)$.
It follows from Step 1b that there is a bijection
\begin{align}
\bigl(\lim_{\nu\to\infty} p\bigl((\psi_\alpha^\nu)^{-1}\bigr)(x_i^\nu)\:|\: 1\leq i \leq r\bigr)
\simeq
\bigl(\lim_{\nu\to\infty} p\bigl((\wt\psi_{g(\alpha)}^\nu)^{-1}\bigr)(x_i^\nu) \:|\: 1 \leq i \leq r\bigr).
\end{align}
It follows from Lemma~\ref{lem:V1_x-limits} that $g(\alpha)$ lies in $V_\comp^{\geq 2}(\wt T_b)$, so we can identify $\incom(\alpha)$ and $\incom(g(\alpha))$.
A similar argument shows that if $\alpha$ lies in $V_\comp^1(T_b)$, then $g(\alpha)$ lies in $V_\comp^1(\wt T_b)$, hence we can identify the incoming neighbor of $\alpha$ with that of $g(\alpha)$.
We have now extended $g$ to a bijection $g\colon V(T_b) \to V(\wt T_b)$.

\medskip
\noindent {\sc Step 1e:} {\it Two vertices $\alpha, \beta$ in $V_\comp(T_b)$ are contiguous if and only if (1) there is no $\gamma \in (\mu_{ij}) \cup \{\mu_\infty^{T_b}\}$ satisfying both $z_{\alpha\gamma} = z_{\alpha\beta}$ and $z_{\beta\gamma} = z_{\beta\alpha}$, and (2) there is no $\delta \in (\lambda_i) \cup \{\lambda_\infty^{T_s}\}$ satisfying both $x_{\alpha\delta} = x_{\alpha\beta}$ and $x_{\beta\delta} = x_{\beta\alpha}$.
Vertices $\alpha \in V_\comp(T_b)$, $\mu_{ij}$ are contiguous if and only if there is no $\gamma \in V_\comp(T_b)$ with $z_{\alpha\mu_{ij}} \neq z_{\alpha\gamma}$.}
\medskip

\noindent Suppose that $\alpha,\beta \in V_\comp(T_b)$ are contiguous, and fix $\gamma \in (\mu_{ij})\cup\{\mu_\infty^{T_b}\}$.
Switching $\alpha$ and $\beta$ if necessary, we may assume that $\gamma$ lies in $(T_b)_{\alpha\beta}$.
Then $z_{\beta\alpha} \neq z_{\beta\gamma}$.
A similar argument produces $\delta \in (\lambda_i) \cup \{\lambda_\infty^{T_s}\}$ with $x_{\beta\alpha} \neq x_{\beta\delta}$.

Next, we prove the contrapositive of the converse: Suppose that $\alpha, \beta \in V_\comp(T)$ are not contiguous, and define $(\alpha = \gamma_1, \gamma_2,\ldots,\gamma_k = \beta)$ to be the vertices in $V_\comp(T_b)$ through which the path from $\alpha$ to $\beta$ passes.
Suppose that $\gamma_2$ lies in $V_\comp^1$.
Define $(\gamma_2 = \delta_1, \delta_2,\ldots, \delta_\ell)$ to be a non-self-intersecting sequence in $V_\comp(T_b) \cup V_\mk(T_b)$ that starts at $\gamma_2$, terminates at a vertex in $(\mu_{ij})_{i,j} \cup \{\alpha_\root^{T_b}\}$, has $\delta_i, \delta_{i+1}$ contiguous for each $i$, and intersects $(\gamma_1,\ldots,\gamma_k)$ only at $\gamma_2$.
(The existence of such a sequence follows from the stability of $2T$.)
Then if we set $\eps \coloneqq \delta_\ell$ if $\delta_\ell \in (\mu_{ij})$ and $\eps \coloneqq \mu_\infty^{T_b}$ if $\delta_\ell = \alpha_\root^{T_b}$, we have $z_{\alpha\eps} = z_{\alpha\beta}$ and $z_{\beta\eps} = z_{\beta\alpha}$.
On the other hand, suppose that $\gamma_2$ lies in $V_\comp^{\geq 2}(T_b)$.
Define $(\pi(\gamma_2)=\rho_1,\rho_2,\ldots,\rho_\ell)$ to be a path in $T_s$ that starts at $\pi(\gamma_2)$, terminates at a vertex in $(\lambda_i)_i\cup \{\rho_\root^{T_s}\}$, and intersects $(\pi(\gamma_1),\ldots,\pi(\gamma_k))$ only at $\pi(\gamma_2)$.
Then if we set $\sigma \coloneqq \rho_\ell$ if $\rho_\ell \in (\lambda_i)$ and $\sigma \coloneqq \lambda_\infty^{T_s}$ if $\rho_\ell = \rho_\root^{T_s}$, we have $x_{\alpha\sigma} = x_{\alpha\beta}$ and $x_{\beta\sigma} = z_{\beta\alpha}$.

A similar, simpler argument proves the second assertion in Step 1e.

\medskip
\noindent {\sc Step 1f:} {\it We show that $f$ extends to an isomorphism of RRTs, then complete Step 1.}
\medskip

It remains to prove the following facts:
\begin{itemize}
\item $g(\alpha_\root) = \wt\alpha_\root$, where we denote $\alpha_\root \coloneqq \alpha_\root^{T_b}$ and $\wt\alpha_\root \coloneqq \alpha_\root^{\wt T_b}$.

\item For $\alpha \in V_\comp(T_b)$, $g$ induces a bijection from $\incom(\alpha)$ to $\incom(g(\alpha))$.

\item For $\alpha, \beta \in V_\comp(T_b)$ with $\beta$ an incoming neighbor of the $i$-th incoming neighbor of $\alpha$, $g(\beta)$ is an incoming neighbor of the $i$-th incoming neighbor of $g(\alpha)$.

\item For $\alpha \in V_\seam(T_b)$ and $\mu_{ij} \in \incom(\alpha)$, $g(\mu_{ij})$ lies in $\incom(g(\alpha))$.

\item $g$ respects the ribbon tree structure of $T_b$ and $\wt T_b$.
\end{itemize}

First, we show $g(\alpha_\root) = \wt\alpha_\root$.
Fix $\wt\alpha \in V_\comp(\wt T_b) \setminus \{g(\alpha_\root)\}$, and write $\wt\alpha = g(\alpha)$ for some $\alpha \in V_\comp(T_b)$.
Step 1c implies $\wt z_{g(\alpha_\root)\wt\alpha} = \chi_{\alpha_\root}^{-1}(z_{\alpha_\root\alpha}) \neq \infty$.
Since $\wt z_{g(\alpha_\root)\wt\alpha}$ is finite for every $\wt\alpha \in V_\comp(\wt T_b)\setminus \{g(\alpha_\root)\}$, we must have $g(\alpha_\root) = \wt\alpha_\root$.

The second bullet is an immediate consequence of the construction of $g$ on $V_\seam(T_b)$.

Next, fix $\alpha, \beta \in V_\comp(T_b)$ with $\beta$ an incoming neighbor of the $i$-th incoming neighbor of $\alpha$.
By Step 1e, there is no $\gamma \in (\mu_{ij}) \cup \{\mu\infty^{T_b}\}$ satisfying both $z_{\alpha\gamma} = z_{\alpha\beta}$ and $z_{\beta\gamma} = z_{\beta\alpha}$, nor is there $\delta \in (\lambda_i) \cup \{\lambda_\infty^{T_s}\}$ satisfying both $x_{\alpha\delta} = x_{\alpha\beta}$ and $x_{\beta\delta} = x_{\beta\alpha}$.
Together with Step 1c, it follows that there is no $\wt\gamma \in (\wt\mu_{ij}) \cup \{\mu_\infty^{\wt T_b}\}$ with both $\wt z_{g(\alpha)\wt\gamma} = \wt z_{g(\alpha)g(\beta)}$ and $\wt z_{g(\beta)\wt\gamma} = \wt z_{g(\beta)g(\alpha)}$, nor is there $\wt\delta \in (\wt\lambda_i) \cup \{\lambda_\infty^{\wt T_s}\}$ with both $\wt x_{g(\alpha)\wt\delta} = \wt x_{g(\alpha)g(\beta)}$ and $\wt x_{g(\beta)\wt\delta} = \wt x_{g(\beta)g(\alpha)}$.
Step 1e now implies that $g(\alpha)$ and $g(\beta)$ are contiguous, and another application of Step 1d implies that $g(\beta)$ is an incoming neighbor of the $i$-th incoming neighbor of $g(\alpha)$.

A similar argument to the previous paragraph shows that for $\alpha \in V_\seam(T_b)$ and $\mu_{ij} \in \incom(\alpha)$, $g(\mu_{ij})$ lies in $\incom(g(\alpha))$.

It follows from Step 1c that for any $\alpha \in V(T_b)$, $g$ induces an order-preserving bijection from $\incom(\alpha)$ to $\incom(g(\alpha))$.

\medskip
\noindent {\sc Step 2:} {\it The general case.}
\medskip

\noindent We begin by noting that for any $\bigl(2T,(\bx_\rho),(\bz_\alpha)\bigr) \in \SWC_\bn$ and $\beta \in V_\comp(T_b)$, we can associate a smooth stable witch curve.
This association depends on whether $\beta$ lies in $V_\comp^1(T_b)$ or $V_\comp^{\geq 2}(T_b)$.
If $\beta$ lies in $V_\comp^1(T_b)$, we associate $\bigl((x_{\beta,1}),\bz_\beta\bigr)$.
Otherwise, we associate $\bigl(\bx_{\pi(\beta)},\bz_\beta\bigr)$.

\medskip
\noindent {\sc Step 2a:} {\it If $\bigl(2T',(\bx_\rho^\nu),(\bz_\alpha^\nu)\bigr)$ Gromov-converges to $\bigl(2T,(\bx_\rho),(\bz_\alpha)\bigr)$ and $(\bx^\nu,\bz^\nu)$ is the sequence of smooth stable witch curves associated as in the previous paragraph to a vertex $\beta \in V_\comp(T_b')$, then $(\bx^\nu,\bz^\nu)$ converges to a restriction of $\bigl(2T,(\bx_\rho),(\bz_\alpha)\bigr)$.}
\medskip

\noindent The only nontrivial part of this step is to spell out which restriction of $2T$ to use.
Denote by $f\colon 2T \to 2T'$ the
stable
tree-pair surjection involved in the Gromov convergence of $\bigl(2T',(\bx_\rho^\nu),(\bz_\alpha^\nu)\bigr)$ to $\bigl(2T,(\bx_\rho),(\bz_\alpha)\bigr)$.
First, suppose $\beta$ lies in $V_\comp^{\geq 2}(T_b')$.
Define a
stable
tree-pair $2T|_\beta$ like so: $T_s|_\beta$ is the preimage under $f_s$ of $\pi(\beta)$ and its incoming neighbors.
$T_b|_\beta$ is the preimage under $f_b$ of $\beta$, its incoming neighbors, and the incoming neighbors of its incoming neighbors.
Then $2T|_\beta$ is a
stable
tree-pair, and it is straightforward to show that the smooth stable witch curves $(\bx^\nu,\bz^\nu)$ associated to $\beta$ Gromov-converge to the restriction of $\bigl(2T,(\bx_\rho),(\bz_\alpha)\bigr)$ to $2T|_\beta$.
The same result can be proven in the case that $\beta$ lies in $V_\comp^1(T_b)$: in this case, set $T_s|_\beta$ to be a single vertex.

\medskip
\noindent {\sc Step 2b:} {\it We establish the general case.}
\medskip

\noindent We are now ready to prove Thm.~\ref{thm:2Mn_unique_limits}.
Since there are only finitely many isomorphism classes of
stable
tree-pairs of type $\bn$, we may pass to a subsequence and assume that $2T^\nu \equiv 2T'$ and that all the
stable
tree-pair surjections $2T \to 2T^\nu$ and $\wt{2T} \to 2T^\nu$ coincide with maps $f\colon 2T \to 2T'$ and $\wt f\colon \wt{2T} \to 2T'$.
Since $\bigl(2T',(\bx_\rho^\nu),(\bz_\alpha^\nu)\bigr)$ Gromov-converges to $\bigl(2T,(\bx_\rho),(\bz_\alpha)\bigr)$, the smooth stable witch curves associated to $\beta$ Gromov-converge to the restriction of $\bigl(2T,(\bx_\rho),(\bz_\alpha)\bigr)$ to $2T|_\beta$, as in Step 2a.
Similarly, $(\bx_{\pi(\beta)}^\nu,\bz_\beta^\nu)$ Gromov-converges to the restriction of $\bigl(\wt{2T},(\wt\bx_{\wt\rho}),(\wt\bz_{\wt\alpha})\bigr)$ to $\wt{2T}_\beta$.
By Step 1, these two restrictions are isomorphic.
Since this holds for every $\beta \in V_\comp(T_b')$, $\bigl(2T,(\bx_\rho),(\bz_\alpha)\bigr)$ and $\bigl(\wt{2T},(\wt\bx_{\wt\rho}),(\wt\bz_{\wt\alpha})\bigr)$ are isomorphic.
\end{proof}

\subsection{The definition and properties of the topology on \texorpdfstring{$\ol{2\cM}_\bn$}{2Mn}}
\label{ss:topology}

Recall that if $X$ is a set and $\cC \subset X \times X^\bN$ is an arbitrary collection of sequences and ``limits'', we can define a topology $\cU(\cC) \subset 2^X$ in which the open sets are those subsets $U \subset X$ having the property that for every $(x_0, (x_n)) \in \cC$ with $x_0 \in U$, $x_n$ is eventually in $U$.
The following lemma gives sufficient conditions for the convergent sequences in $\cU(\cC)$ to coincide with $\cC$.

\begin{lemma}[Lemma 5.6.5, \cite{ms:jh}]
\label{lem:2Mn_mu_top}
Let $X$ be a set and $\cC \subset X \times X^\bN$ be a collection of sequences in $X$ that satisfies the property that if $(x_0, (x_n)_n) \in \cC$ and $(y_0,(x_n)_n) \in \cC$, then $x_0 = y_0$.
Suppose that for every $x \in X$ there exists a constant $\eps_0(x) >0$ and a collection of functions $X \to [0,\infty]\colon x' \mapsto \mu_\eps(x,x')$ for $0 < \eps <\eps_0(x)$ satisfying the following conditions.
\begin{itemize}
\item[(a)] If $x \in X$ and $0 < \eps < \eps_0(x)$, then $\mu_\eps(x,x) = 0$.

\item[(b)] If $x \in X$, $0<\eps<\eps_0(x)$, and $(x_n)_n \in X^\bN$, then
\begin{align}
(x,(x_n)_n) \in \cC \iff \lim_{n\to\infty} \mu_\eps(x,x_n) = 0.	
\end{align}

\item[(c)] If $x \in X$, $0<\eps<\eps_0(x)$, and $(x',(x_n)_n) \in \cC$, then
\begin{align}
\mu_\eps(x,x') < \eps \implies \limsup_{n\to\infty} \mu_\eps(x,x_n) \leq \mu_\eps(x,x').	
\end{align}
\end{itemize}
Then $\cC = \cC(\cU(\cC))$.
Moreover, the topology $\cU(\cC)$ is first countable and Hausdorff.
\end{lemma}

We will construct a topology on $\ol{2\cM}_\bn$ by using this lemma.
To begin, we define the functions
$\mu_\eps(x,-)\colon \ol{2\cM}_\bn \to [0,\infty]$.

\begin{definition}
For any two stable witch curves $\bigl(2T, (\bx_\rho), (\bz_\alpha)\bigr)$, $\bigl(\wt{2T}, (\wt\bx_{\wt\rho}), (\wt\bz_{\wt\alpha})\bigr)$ of type $\bn$ and for any $\eps > 0$, define a nonnegative real number $\mu_\eps\bigl(\bigl(2T,(\bx_\rho),(\bz_\alpha)\bigr),\bigl(\wt{2T}, (\wt\bx_{\wt\rho}), (\wt\bz_{\wt\alpha})\bigr)\bigr)$ like so:
\begin{align*}
&\mu_\eps\bigl(\bigl(2T,(\bx_\rho),(\bz_\alpha)\bigr),\bigl(\wt{2T}, (\wt\bx_{\wt\rho}), (\wt\bz_{\wt\alpha})\bigr)\bigr)
 \\
&\hspace{1.5in} \coloneqq
\min_{2f\colon 2T \to \wt{2T}} \inf_{
{(\phi_\rho)_{\rho\in V_\inte(T_s)}}
\atop
{(\psi_\alpha)_{\alpha \in V_\comp(T_b)}}}\!\!\! \mu_\eps\bigl(\bigl(2T, (\bx_\rho),(\bx_\alpha)\bigr),\bigl(\wt{2T}, (\wt\bx_{\wt\rho}),(\wt\bz_{\wt\alpha})\bigr); 2f, (\phi_\rho),(\psi_\alpha)\bigr),
\\
&\mu_\eps\bigl(\bigl(2T, (\bx_\rho), (\bz_\alpha)\bigr),\bigl(\wt{2T}, (\wt\bx_{\wt\rho}), (\wt\bz_{\wt\alpha})\bigr); 2f, (\phi_\rho), (\psi_\alpha)\bigr) \nonumber \\
&\hspace{0.25in}\coloneqq
\sum_{{\rho,\sigma \in V_\inte(T_s), \rho\neq\sigma,} \atop
{f_s(\rho) = f_s(\sigma)}} \sup_{(\bR\cup\{\infty\}) \setminus B_\eps(x_{\rho\sigma})} \!\!\! d\bigl(\phi_\sigma^{-1}\circ\phi_\rho, x_{\sigma\rho}\bigr)
+\!\!\!
\sum_{{\alpha,\beta \in V_\comp(T_b), \alpha\neq\beta,} \atop
{f_b(\alpha) = f_b(\beta)}} \sup_{(\bR^2\cup\{\infty\}) \setminus B_\eps(z_{\alpha\beta})} \!\!\!\!\! d\bigl(\psi_\beta^{-1}\circ\psi_\alpha, z_{\beta\alpha}\bigr) \nonumber \\
&\hspace{0.75in}
+ \sum_{{\rho \in V_\inte(T_s),\sigma \in V(T_s),} \atop
{f_s(\rho) \neq f_s(\sigma)}} d\bigl(\phi_\rho^{-1}(\wt x_{f_s(\rho)f_s(\sigma)}), x_{\rho\sigma}\bigr)
+ \sum_{{{{\alpha \in V_\comp(T_b),}
\atop
{\beta \in V_\comp(T_b)\cup (\mu_{ij})_{i,j},}}} \atop
{f_b(\alpha) \neq f_b(\beta)}} d\bigl(\psi_\alpha^{-1}(\wt z_{f_b(\alpha)f_b(\beta)}), z_{\alpha\beta}\bigr),
\nonumber
\end{align*}
where in the first line we take the minimum over all
stable
tree-pair surjections $2f\colon 2T \to \wt{2T}$ and the infimum over all tuples $(\phi_\rho) \subset G_1$ and $(\psi_\alpha) \subset G_2$ satisfying $p(\psi_\alpha) = \phi_{\pi(\alpha)}$ for every $\alpha \in V_\comp^{\geq 2}(T_b)$, and where in the second line we use the distance metrics on $\bR\cup\{\infty\}$ and $\bR^2\cup\{\infty\}$ induced by identifying these spaces with round spheres.
By convention, if there is no
stable tree-pair surjection $2T \to \wt{2T}$, we set $\mu_\eps\bigl(\bigl(2T,(\bx_\rho),(\bz_\alpha)\bigr),\bigl(\wt{2T},(\wt \bx_{\wt\rho}),(\wt \bz_{\wt\alpha})\bigr)\bigr) \coloneqq \infty$.
\null\hfill$\triangle$
\end{definition}

\begin{remark}
\label{rmk:mu_descends}
It is an immediate consequence of the definition that for any $\bigl(2T,(\bx_\rho),(\bz_\alpha)\bigr)$, \\$\mu_\eps\bigl(\bigl(2T,(\bx_\rho),(\bz_\alpha)\bigr),-\bigr)$ descends to $\ol{2\cM}_\bn$.
Our aim is to use the $\mu_\eps$'s to define a topology on $\ol{2\cM}_\bn$ via Lemma~\ref{lem:2Mn_mu_top}, so we now extend the definition of $\mu_\eps$ to a collection of functions
\begin{align}
\mu_\eps\bigl(\bigl[2T,(\bx_\rho),(\bz_\alpha)\bigr],-\bigr)\colon \ol{2\cM}_\bn \to [0,\infty]
\end{align}
like so: for any $\bigl[2T,(\bx_\rho),(\bz_\alpha)\bigr] \in \ol{2\cM}_\bn$,
choose a representative $\bigl(2T,(\bx_\rho),(\bz_\alpha)\bigr)$.
Now define
\begin{align}
\mu_\eps\bigl(\bigl[2T,(\bx_\rho),(\bz_\alpha)\bigr],-\bigr)
\coloneqq
\mu_\eps\bigl(\bigl(2T,(\bx_\rho),(\bz_\alpha)\bigr),-\bigr).
\end{align}
We note that Lemma~\ref{lem:2Mn_mu_top} does not require $\mu_\eps$ to have any continuity in its first argument.
\end{remark}

\begin{remark}
The quantity $\mu_\eps$ should be compared with a similar quantity, $\rho_\eps$, which plays the analogous role in the definition of the
Grothendieck--Knudsen
topology on $\ol\cM_r$.
(Compare also the analogous quantity used in \S5, \cite{ms:jh} to define the topology on the space of stable maps.)
For any two stable disk trees $\bigl(T,(\bx_\rho)\bigr), \bigl(\wt T,(\wt\bx_{\wt\rho})\bigr)$ with $r$ leaves and for $\eps > 0$, $\rho_\eps\bigl(\bigl(T,(\bx_\rho)\bigr), \bigl(\wt T,(\wt\bx_{\wt\rho})\bigr)\bigr)$ is defined like so:
\begin{align}
\rho_\eps\bigl(\bigl(T,(\bx_\rho)\bigr), \bigl(\wt T,(\wt\bx_{\wt\rho})\bigr)\bigr)
&\coloneqq
\min_{f\colon T \to \wt T} \inf_{(\phi_\rho)_{\rho\in V_\inte(T)}} \rho_\eps\bigl(\bigl(T,(\bx_\rho)\bigr), \bigl(\wt T,(\wt\bx_{\wt\rho})\bigr); f, (\phi_\rho)\bigr), \\
\rho_\eps\bigl(\bigl(T,(\bx_\rho)\bigr), \bigl(\wt T,(\wt\bx_{\wt\rho})\bigr); f, (\phi_\rho)\bigr) &\coloneqq
\sum_{
{\rho,\sigma \in V_\inte(T_s),\rho\neq\sigma,}
\atop
{f(\rho)= f(\sigma)}}
\sup_{(\bR\cup\{\infty\})\setminus B_\eps(x_{\rho\sigma})}
d\bigl(\phi_\sigma^{-1}\circ\phi_\rho,x_{\sigma\rho}\bigr)
 \nonumber \\
&\hspace{1.5in}
+
\sum_{
{\rho \in V_\inte(T),\sigma \in V(T),}
\atop
{f(\rho)\neq f(\sigma)}} d\bigl(\phi_\rho^{-1}(\wt x_{f(\rho)f(\sigma)}),x_{\rho\sigma}\bigr),
\nonumber
\end{align}
where in the first line we take the minimum over all surjections $f\colon T \to \wt T$ between stable RRTs.
\end{remark}

\begin{lemma}
\label{lem:2Mn_mu_props}
Fix $\bigl(2T,(\bx_\rho),(\bz_\alpha)\bigr) \in \SWC_\bn$.
Then the following hold for every $\eps > 0$:
\begin{itemize}
\item[] {\sc (convergence)} A sequence $\bigl(2T^\nu,(\bx_\rho^\nu),(\bz_\alpha^\nu)\bigr) \subset \SWC_\bn$ Gromov-converges to $\bigl(2T,(\bx_\rho),(\bz_\alpha)\bigr)$ if and only if $\mu_\eps\bigl(\bigl(2T, (\bx_\rho),(\bz_\alpha)\bigr),\bigl(2T^\nu, (\bx_\rho^\nu),(\bz_\alpha^\nu)\bigr)\bigr)$ converges to 0.

\item[] {\sc (triangle)} If $\bigl(\wt{2T}, (\wt\bx_{\wt\rho}),(\wt\bz_{\wt\alpha})\bigr) \in \SWC_\bn$ satisfies $\mu_\eps\bigl(\bigl(2T, (\bx_\rho),(\bz_\alpha)\bigr),\bigl(\wt{2T}, (\wt\bx_{\wt\rho}),(\wt\bz_{\wt\alpha})\bigr)\bigr) < \eps$ and the sequence $\bigl(2T^\nu,(\bx_\rho^\nu),(\bz_\alpha^\nu)\bigr)$ Gromov-converges to $\bigl(\wt{2T},(\wt\bx_{\wt\rho}),(\wt\bz_{\wt\alpha})\bigr)$, then
\begin{align}
\limsup_{\nu\to\infty} \mu_\eps\bigl(\bigl(2T, (\bx_\rho),(\bz_\alpha)\bigr),\bigl(2T^\nu, (\bx_\rho^\nu),(\bz_\alpha^\nu)\bigr)\bigr)	 \leq \mu_\eps\bigl(\bigl(2T, (\bx_\rho),(\bz_\alpha)\bigr),\bigl(\wt{2T}, (\wt\bx_{\wt\rho}),(\wt\bz_{\wt\alpha})\bigr)\bigr).
\end{align}
\end{itemize}	
\end{lemma}

\begin{proof}
\begin{itemize}
\item[] {\sc (convergence)} If $\bigl(2T^\nu,(\bx_\rho^\nu),(\bz_\alpha^\nu)\bigr)$ Gromov-converges to $\bigl(2T,(\bx_\rho),(\bz_\alpha)\bigr)$ then it follows from {\sc (rescaling')} and {\sc (special point')}, and the analogous properties for Gromov convergence of stable disk trees, that $\mu_\eps\bigl(\bigl(2T, (\bx_\rho),(\bz_\alpha)\bigr),\bigl(T^\nu, (\bx_\rho^\nu),(\bz_\alpha^\nu)\bigr)\bigr)$ converges to 0.

Conversely, suppose $\mu_\eps\bigl(\bigl(2T, (\bx_\rho),(\bz_\alpha)\bigr),\bigl(2T^\nu, (\bx_\rho^\nu),(\bz_\alpha^\nu\bigr)\bigr)$ converges to 0.
Then it is the case that for $\nu$ large enough there is a
stable
tree-pair surjection $2f^\nu\colon 2T \to 2T^\nu$ and tuples $(\phi_\rho^\nu)$, $(\psi_\alpha^\nu)$ with $p(\psi_\alpha) = \phi_{\pi(\alpha)}$ for $\alpha \in V_\comp^{\geq2}(T_b)$ such that the following inequality holds:
\begin{align}
&\mu^\nu \coloneqq \mu_\eps\bigl(\bigl(2T,(\bx_\rho),(\bz_\alpha)\bigr),\bigl(2T^\nu,(\bx_\rho^\nu),(\bz_\alpha^\nu)\bigr); 2f^\nu, (\phi_\rho^\nu), (\psi_\alpha^\nu)\bigr) \nonumber \\
&\hspace{3in} \leq
\mu_\eps\bigl(\bigl(2T,(\bx_\rho),(\bz_\alpha)\bigr),\bigl(2T^\nu,(\bx_\rho^\nu),(\bz_\alpha^\nu)\bigr)\bigr) + 2^{-\nu}.
\nonumber	
\end{align}
Since there are only finitely many
stable
tree-pair surjections with domain $2T$, we may assume all the
stable
tree-pairs $2T^\nu$ are equal to a single $\wt{2T}$ and all the maps $2f^\nu\colon 2T \to \wt{2T}$ are equal to a single $2f$.
First, we verify {\sc (rescaling)}.
Fix contiguous $\alpha,\beta \in V_\comp(T_b)$ with $f_b(\alpha) = f_b(\beta)$; without loss of generality we may assume $\alpha$ is closer to the root than $\beta$, so $z_{\alpha\beta} \in \bR^2$ and $z_{\beta\alpha} = \infty$.
The convergence $\mu^\nu \to 0$ implies that $(\psi_\alpha^\nu)^{-1}\circ\psi_\beta^\nu$ converges to $z_{\alpha\beta}$ uniformly on $\bR^2 \setminus B_\eps(\infty)$, hence $(\psi_\alpha^\nu)^{-1} \circ \psi_\beta^\nu$ converges to $z_{\alpha\beta}$ u.c.s.\ away from $\infty$.
From this it follows that $(\psi_\beta^\nu)^{-1} \circ \psi_\alpha^\nu$ converges to $\infty$ u.c.s.\ away from $z_{\alpha\beta}$, so we have established {\sc (rescaling)}.
The {\sc (special point)} requirement obviously holds.
Finally, the inequality
\begin{align}
\rho_\eps\bigl(\bigl(T_s,(\bx_\rho)\bigr),\bigl(T_s^\nu,(\bx_\rho^\nu)\bigr);f_s^\nu,(\phi_\rho^\nu)\bigr)
\leq
\mu_\eps\bigl(\bigl(2T,(\bx_\rho),(\bz_\alpha)\bigr),\bigl(2T^\nu,(\bx_\rho^\nu),(\bz_\alpha^\nu)\bigr);2f^\nu,(\phi_\rho^\nu),(\psi_\alpha^\nu)\bigr)
\end{align}
implies that the left-hand side converges to 0, so $\bigl(T_s^\nu,(\bx_\rho^\nu)\bigr)$ Gromov-converges to $\bigl(T_s,(\bx_\rho)\bigr)$ via $f_s^\nu$ and $(\phi_\rho^\nu)$.
(This uses the analogue of the current lemma for $\SDT_r$.)
We may conclude that $\bigl(2T^\nu, (\bx_\rho^\nu), (\bz_\alpha^\nu)\bigr)$ Gromov-converges to $\bigl(2T,(\bx_\rho),(\bz_\alpha)\bigr)$.

\medskip

\item[] {\sc (triangle)} The inequality $\mu_\eps\bigl(\bigl(2T, (\bx_\rho),(\bz_\alpha)\bigr),\bigl(\wt{2T}, (\wt\bx_{\wt\rho}),(\wt\bz_{\wt\alpha})\bigr)\bigr) < \eps$ implies that there exists a
stable
tree-pair surjection $2g\colon 2T \to \wt{2T}$ and tuples $(\chi_\rho)_{\rho \in V_\inte(T_s)} \subset G_1$, $(\xi_\alpha)_{\alpha\in V_\comp(T_b)} \subset G_2$ with
$p(\xi_\alpha) = \chi_{\pi(\alpha)}$
for $\alpha \in V_\comp^{\geq 2}(T_b)$ such that
\begin{align} \label{eq:metric_bound}
\mu_\eps\bigl(\bigl(2T, (\bx_\rho),(\bz_\alpha)\bigr),\bigl(\wt{2T}, (\wt\bx_{\wt\rho}),(\wt\bz_{\wt\alpha})\bigr), 2g, (\chi_\rho),(\xi_\alpha)\bigr) < \eps.
\end{align}
It follows that for every pair $\alpha, \beta \in V_\comp(T_b)$ with $g_b(\alpha) \neq g_b(\beta)$ we have
\begin{align} \label{eq:nodal_bound1}
d\bigl(\xi_\alpha^{-1}(\wt z_{g_b(\alpha)g_b(\beta)}), z_{\alpha\beta}\bigr) < \eps;
\end{align}
similarly, for every $\rho, \sigma \in V_\inte(T_s)$ with $g_s(\alpha) \neq g_s(\beta)$ we have
\begin{align} \label{eq:nodal_bound2}
d\bigl(\chi_\rho^{-1}(\wt x_{g_s(\rho)g_s(\sigma)}),x_{\rho\sigma}\bigr) < \eps.
\end{align}
Now suppose that $\bigl(2T^\nu, (\bx_\rho^\nu),(\bz_\alpha^\nu)\bigr)$ Gromov-converges to $\bigl(\wt{2T}, (\wt\bx_{\wt\rho}),(\wt\bz_{\wt\alpha})\bigr)$ via
stable
tree-pair surjections $2f^\nu\colon \wt{2T} \to 2T^\nu$ and reparametrizations $(\phi_{\wt\rho}^\nu)$ and $(\psi_{\wt\alpha}^\nu)$.
To prove {\sc (triangle)}, it suffices to prove the following equality:
\begin{align} \label{eq:triangle_equiv}
&\mu_\eps\bigl(\bigl(2T,(\bx_\rho),(\bz_\alpha)\bigr),\bigl(\wt{2T},(\wt\bx_{\wt\rho}),(\wt\bz_{\wt\alpha})\bigr); 2g, (\chi_\rho), (\xi_\alpha)\bigr) \\
&\hspace{0.5in} = \lim_{\nu\to\infty} \mu_\eps\bigl(\bigl(2T,(\bx_\rho),(\bz_\alpha)\bigr),\bigl(2T^\nu,(\bx_\rho^\nu),(\bz_\alpha^\nu)\bigr); 2f^\nu\circ 2g, (\phi_{g_s(\rho)}^\nu\circ\chi_\rho),(\psi_{g_b(\alpha)}^\nu\circ\xi_\alpha)\bigr). \nonumber
\end{align}
Since there are only finitely many
stable
tree-pair surjections with domain $\wt{2T}$, we may assume $2T^\nu \equiv 2T'$ and $2f^\nu \equiv 2f \colon \wt{2T} \to 2T'$.
For any distinct $\alpha,\beta \in V_\comp(T_b)$ with $g_b(\alpha) = g_b(\beta)$ we have $\xi_\beta^{-1} \circ \xi_\alpha = (\psi_{g_b(\beta)}^\nu\circ\xi_\beta)^{-1} \circ (\psi_{g_b(\alpha)}^\nu\circ\xi_\alpha)$, hence
\begin{align} \label{eq:triangle_helper1}
\sup_{(\bR^2\cup\{\infty\})\setminus B_\eps(z_{\alpha\beta})} d\bigl(\xi_\beta^{-1}\circ\xi_\alpha,z_{\beta\alpha}\bigr) = \sup_{(\bR^2\cup\{\infty\})\setminus B_\eps(z_{\alpha\beta})} d\bigl((\psi_{g_b(\beta)}^\nu\circ\xi_\beta)^{-1} \circ (\psi_{g_b(\alpha)}^\nu\circ\xi_\alpha),z_{\beta\alpha}\bigr).
\end{align}
Similarly, for distinct $\rho,\tau \in V_\inte(T_s)$ with $g_s(\rho)=g_s(\tau)$, we have
\begin{align} \label{eq:triangle_helper2}
\sup_{(\bR\cup\{\infty\})\setminus B_\eps(x_{\rho\sigma})} d\bigl(\chi_\sigma^{-1}\circ\chi_\rho,x_{\sigma\rho}\bigr) = \sup_{(\bR\cup\{\infty\})\setminus B_\eps(x_{\rho\sigma})} d\bigl((\phi_{g_s(\sigma)}^\nu\circ\chi_\sigma)^{-1} \circ (\phi_{g_s(\rho)}^\nu\circ\chi_\rho),x_{\sigma\rho}\bigr).
\end{align}
If $\alpha, \beta \in V_\comp(T_b)$ have $g_b(\alpha) \neq g_b(\beta)$ and $f_b^\nu(g_b(\alpha)) = f_b^\nu(g_b(\beta))$, then {\sc (rescaling')} implies that $(\psi_{g_b(\beta)}^\nu)^{-1} \circ \psi_{g_b(\alpha)}^\nu$ converges to $\wt z_{g_b(\beta)g_b(\alpha)}$ u.c.s.\ away from
$\wt z_{g_b(\alpha)g_b(\beta)}$,
hence by \eqref{eq:nodal_bound1} $(\psi_{g_b(\beta)}^\nu\circ\xi_\beta)^{-1} \circ (\psi_{g_b(\alpha)}^\nu\circ\xi_\alpha)$ converges to $\xi_\beta^{-1}(\wt z_{g_b(\beta)g_b(\alpha)})$ uniformly on $(\bR^2\cup\{\infty\}) \setminus B_\eps(z_{\alpha\beta})$.
We therefore have
\begin{align} \label{eq:triangle_helper3}
d\bigl(\xi_\beta^{-1}(\wt z_{g_b(\beta)g_b(\alpha)}), z_{\beta\alpha}\bigr) = \lim_{\nu\to\infty} \sup_{(\bR^2\cup\{\infty\})\setminus B_\eps(z_{\alpha\beta})} d\bigl((\psi_{g_b(\beta)}^\nu\circ\xi_\beta)^{-1} \circ (\psi_{g_b(\alpha)}^\nu\circ\xi_\alpha),z_{\beta\alpha}\bigr).
\end{align}
Similarly, it follows from \eqref{eq:nodal_bound2} that if $\rho,\sigma \in V_\inte(T_s)$ have $g_s(\alpha) \neq g_s(\sigma)$ and $f_s^\nu(g_s(\rho)) = f_s^\nu(g_s(\sigma))$, we have
\begin{align} \label{eq:triangle_helper4}
d\bigl(\chi_\sigma^{-1}(\wt z_{g_s(\sigma)g_s(\rho)}), x_{\sigma\rho}\bigr) = \lim_{\nu\to\infty} \sup_{(\bR\cup\{\infty\})\setminus B_\eps(x_{\rho\sigma})} d\bigl((\phi_{g_s(\sigma)}^\nu \circ \chi_\sigma)^{-1} \circ (\phi_{g_s(\rho)}^\nu \circ \chi_\rho),x_{\sigma\rho}\bigr).
\end{align}
Finally, if
$\alpha \in V_\comp(T_b) \cup V_\mk(T_b)$, $\beta \in V_\comp(T_b)$
have $f_b^\nu(g_b(\alpha)) \neq f_b^\nu(g_b(\beta))$, then {\sc (special point')} implies the convergence of $(\wt\psi_{g_b(\beta)}^\nu)^{-1}(z_{f_b^\nu(g_b(\beta))f_b^\nu(g_b(\alpha))}^\nu)$ to $\wt z_{g_b(\beta)g_b(\alpha)}$, hence
\begin{align} \label{eq:triangle_helper5}
d\bigl(\xi_\beta^{-1}(\wt z_{g_b(\beta)g_b(\alpha)}),z_{\beta\alpha}\bigr) = \lim_{\nu\to\infty} d\bigl((\psi_{g_b(\beta)}^\nu\circ\xi_\beta)^{-1}(z_{f_b^\nu(g_b(\beta))f_b^\nu(g_b(\alpha))}^\nu), z_{\beta\alpha}\bigr).
\end{align}
Similarly, if
$\rho \in V(T_s)$, $\sigma \in V_\inte(T_s)$
have $f_s^\nu(g_s(\rho)) \neq f_s^\nu(g_s(\sigma))$, then we have
\begin{align} \label{eq:triangle_helper6}
d\bigl(\chi_\sigma^{-1}(\wt x_{g_s(\sigma)g_s(\rho)}),x_{\sigma\rho}\bigr) = \lim_{\nu\to\infty} d\bigl((\phi_{g_s(\sigma)}^\nu\circ\chi_\sigma)^{-1}(x_{f_s^\nu(g_s(\sigma))f_s^\nu(g_s(\rho))}^\nu), z_{\sigma\rho}\bigr).
\end{align}
\eqref{eq:triangle_helper1}, \eqref{eq:triangle_helper2},\eqref{eq:triangle_helper3},
    \eqref{eq:triangle_helper4},
\eqref{eq:triangle_helper5}, and \eqref{eq:triangle_helper6} together yield \eqref{eq:triangle_equiv},
by showing that each term on the left-hand side of \eqref{eq:triangle_equiv} is equal to a corresponding term on the right-hand side.
\end{itemize}
\end{proof}

We now define the
{\bf{Grothendieck--Knudsen topology}}
on $\ol{2\cM}_\bn$ to be $\cU(\cC)$, where $\cC$ are the Gromov-convergent sequences.
Moreover, we equip $\ol{2\cM}_\bn$ with the $W_\bn$-stratification defined by sending $\bigl(2T,(\bx_\rho),(\bz_\alpha)\bigr)$ to $2T$.
It is immediate from the definition of Gromov convergence that this map is continuous with respect to the Alexandroff topology on $W_\bn$.

\begin{proof}[Proof of Thm.~\ref{thm:main}]
It follows from Thm.~\ref{thm:2Mn_unique_limits} that Gromov-convergent sequences have unique limits.
This, together with Rmk.~\ref{rmk:mu_descends} and Lemma~\ref{lem:2Mn_mu_props}, imply that Gromov-convergent sequences satisfy the hypotheses of Lemma~\ref{lem:2Mn_mu_top}.
This proves that convergence in the
Grothendieck--Knudsen
topology on $\ol{2\cM}_\bn$ is equivalent to Gromov convergence, and that $\ol{2\cM}_\bn$ is first-countable and Hausdorff.

The rest of the proof of the topological properties of $\ol{2\cM}_\bn$ hinges on showing that $\ol{2\cM}_\bn$ is second-countable, just as the analogous result for $\ol\cM_r$ depends similarly on showing that $\ol\cM_r$ is second-countable.
The proof of this result for $\ol\cM_r$ in \S 5, \cite{ms:jh} contains a gap; McDuff--Salamon have communicated to the author a fix, which they intend to include in future editions of \cite{ms:jh}.
This fix applies equally well to the current proof.

Finally, we observe that the forgetful map $W_\bn \to K_r$ extends to a map $\ol{2\cM}_\bn \to \ol\cM_r$, sending $\bigl(2T,(\bx_\rho),(\bz_\alpha)\bigr)$ to $\bigl(T_s,(\bx_\rho)\bigr)$.
This map sends Gromov-convergent sequences to Gromov-convergent sequences, and $\ol{2\cM}_\bn$ is first-countable, so this map is continuous.
\end{proof}

\end{document}